\documentclass[12pt]{article}

\usepackage[margin=1.3in, top=1.3in, bottom=1.3in]{geometry}

\usepackage[utf8]{inputenc} 
\usepackage[T1]{fontenc}    
\usepackage{url}            
\usepackage{booktabs}       
\usepackage{amsfonts}       
\usepackage{nicefrac}       
\usepackage{microtype}      

\usepackage{amsmath,amssymb,amsthm,xcolor}
\usepackage{array}
\usepackage{float}
\newtheorem{theorem}{Theorem}[section]

\newtheorem{lemma}[theorem]{Lemma}

\newtheorem{proposition}[theorem]{Proposition}

\newtheorem{fact}[theorem]{Fact}
\usepackage{graphicx}
\usepackage{subcaption}
\usepackage{mdframed}
\usepackage{lipsum}
\usepackage{wrapfig}
\usepackage[hidelinks,hypertexnames=false]{hyperref}
\usepackage{tikz}
\usepackage{enumitem}
\allowdisplaybreaks
\usepackage{epstopdf}
\usepackage{algorithmic}
\usepackage{booktabs}
\usepackage{changepage}

\title{\LARGE Convergence of the momentum method for semialgebraic functions with locally Lipschitz gradients}

\begin{document}

\author{\large C\'edric Josz\thanks{\url{cj2638@columbia.edu}, IEOR, Columbia University, New York. Research supported by NSF EPCN grant 2023032 and ONR grant N00014-21-1-2282.} \and Lexiao Lai\thanks{\url{ll3352@columbia.edu}, IEOR, Columbia University, New York.} \and Xiaopeng Li\thanks{\url{xl3040@columbia.edu}, IEOR, Columbia University, New York.}}
\date{}

\maketitle

\begin{center}
    \textbf{Abstract}
    \end{center}
    \vspace*{-3mm}
 \begin{adjustwidth}{0.2in}{0.2in}
 
~~~~We propose a new length formula that governs the iterates of the momentum method when minimizing differentiable semialgebraic functions with locally Lipschitz gradients. It enables us to establish local convergence, global convergence, and convergence to local minimizers without assuming global Lipschitz continuity of the gradient, coercivity, and a global growth condition, as is done in the literature. As a result, we provide the first convergence guarantee of the momentum method starting from arbitrary initial points when applied to matrix factorization, matrix sensing, and linear neural networks.

\end{adjustwidth} 
\vspace*{3mm}

\section{Introduction} 
\label{sec:Introduction}

The gradient method with constant momentum and constant step size (or momentum method for short \cite[Equation (10)]{kovachki2021continuous}) for minimizing a differentiable function $f: \mathbb{R}^n\rightarrow\mathbb{R}$ consists in choosing initial points $x_{-1},x_0\in \mathbb{R}^n$ and generating a sequence of iterates according to the update rule
\begin{equation}
\label{eq:momentum}
    x_{k+1} = x_k + \beta(x_k-x_{k-1}) -\alpha \nabla f (x_k + \gamma(x_k-x_{k-1})),  ~~~ \forall k \in \mathbb{N},
\end{equation}
where $\alpha>0$ is the step size and $\beta \in (-1,1)$ and $\gamma \in \mathbb{R}$ are constant momentum parameters, as implemented in PyTorch \cite{paszke2017automatic,paszke2019pytorch} and TensorFlow \cite{abadi2016tensorflow}. When $\gamma = 0$, this reduces to Polyak's heavy ball method \cite[Equation (9)]{polyak1964some} and when $\beta = \gamma$, this reduces to Nesterov's accelerated gradient method \cite[Equation (2.2.22)]{nesterov2018introductory}. If the objective function is strongly convex and satisfies some regularity assumptions, the former has a nearly optimal local convergence rate \cite[Theorem 9]{polyak1964some} \cite[Theorem 2.1.13]{nesterov2018introductory}, while the latter has a globally optimal convergence rate \cite[Equation (2.2.23)]{nesterov2018introductory}. This holds with a suitable choice of parameters $\alpha$, $\beta$, and $\gamma$.

Various objective functions of interest nowadays are however not convex, including matrix factorization, matrix sensing, and linear neural networks, respectively given by
\begin{subequations}
    \begin{gather}
        (X,Y)\in \mathbb{R}^{m\times r} \times \mathbb{R}^{n\times r} \mapsto \|XY^\top - M\|_F^2, \label{eq:pb1}\\ 
        (X,Y)\in \mathbb{R}^{m\times r} \times \mathbb{R}^{n\times r} \mapsto \sum\limits_{i=1}^m(\langle A_i,XY^\top\rangle_F-b_i)^2, \label{eq:pb2}\\
        (W_1,\hdots,W_l) \in \mathbb{R}^{n_1\times n_0} \times \hdots \times \mathbb{R}^{n_l\times n_{l-1}} \mapsto  \|W_l \hdots W_1\bar{X} - \bar{Y}\|^2_F. \label{eq:pb3}
    \end{gather}
\end{subequations}
Above, $M,A_1,\hdots,A_m \in \mathbb{R}^{m\times n}$, $b_1,\hdots,b_m\in \mathbb{R}$, $\bar{X} \in \mathbb{R}^{n_0\times n}$, $\bar{Y} \in \mathbb{R}^{n_l\times n}$, and $\|\cdot\|_F = \sqrt{\langle \cdot,\cdot\rangle_F}$ is the Frobenius norm. These problems have in common that they are semialgebraic and have locally Lipschitz continuous gradients. However, they do not have globally Lipschitz continuous gradients, they are not coercive, and whether they satisfy a global growth condition is unknown and hard to check for. In other words, the commonly used assumptions H1 (sufficient decrease), H2 (relative error), H3 (continuity) due to Attouch \text{et al.} \cite{attouch2013convergence}, adapted to the momentum method in \cite{ochs2014ipiano,ochs2018local}, are not true, and H4 (global growth) is unknown. As a result, it is not known whether the momentum method --- in particular the heavy ball method and Nesterov’s accelerated gradient method --- would converge if the initial points lie close to a local minimizer of $f$. A fortiori, nothing is known if they are chosen arbitrarily or at random in $\mathbb{R}^n$. 

Even if one assumes that the iterates are bounded, a common assumption in the literature which implies H3, it is not known whether the iterates would converge. Indeed, choose a step size $\alpha>0$ and suppose that the iterates are bounded. Let $L>0$ denote a Lipschitz constant of the gradient on the convex hull of the iterates. If $\alpha \geqslant 2/L$, then the argument employed in \cite[Theorem 3.2]{absil2005convergence} and \cite[Theorem 3.2]{attouch2013convergence}, which consists of taking a subsequence and invoking the Kurdyka-\L{}ojasiewicz inequality \cite[Proposition 1 p. 67]{law1965ensembles} \cite[Theorem 1]{kurdyka1998gradients}, fails to establish convergence. Unfortunately, there is no way to control the size of $L$ before choosing the step size $\alpha$. 

We next review the literature on the momentum method in the nonconvex setting. All the results in the literature require that $f$ has an $L$-Lipschitz continuous gradient with $L>0$, along with other assumptions that we next describe. First, we discuss convergence when the initial points are near a local minimizer. If the objective function $f$ satisfies the Kurdyka-\L{}ojasiewicz inequality at a local minimizer $x^*\in \mathbb{R}^n$, $\alpha \in (0,2(1-\beta)/L)$, $\beta \in [0,1)$, $\gamma = 0$, and a global growth condition \cite[(H4)]{ochs2018local} is satisfied, then the momentum method converges to a local minimizer when initialized sufficiently close to $x^*$ \cite[Theorem 3.2]{ochs2018local}. The growth condition implies the existence of constants $a,b>0$ such that $f(x) + b\|x-y\|^2 \geqslant f(x^*) - a\|y-x^*\|^2$ for all $x,y\in\mathbb{R}^n$, and in particular $f(x) \geqslant f(x^*) - a\|x-x^*\|^2$ for all $x\in\mathbb{R}^n$. 

Second, we discuss convergence when the initial points are arbitrary. Under the same parameter settings for $\alpha,\beta,\gamma$, if $f$ is lower bounded, then the gradients $\nabla f(x_k)$ converge to zero \cite[Lemmas 1,2,3]{zavriev1993heavy} for any initial points $x_{-1},x_0 \in \mathbb{R}^n$. If in addition the function is coercive and satisfies the Kurdyka-\L{}ojasiewicz inequality \cite{kurdyka1998gradients} at every point and $x_{-1}=x_0$, then the iterates have finite length \cite[Theorem 4.9]{ochs2014ipiano}. If the \L{}ojasiewicz gradient inequality holds \cite[Proposition 1 p. 67]{law1965ensembles}, then a local convergence rate can be deduced \cite[Theorem 3.3]{ochs2018local}. If instead the function satisfies an error bound and its level sets are properly separated, then with $\alpha \in (0,1/L)$, $\beta = \gamma \in [0,1/\sqrt{1+L\alpha})$, and $x_{-1}=x_0$, the iterates and the function values converge linearly to a critical point and a critical value respectively \cite[Theorem 3.7]{wen2017linear}. Finally, if the function satisfies the Kurdyka-\L{}ojasiewicz inequality and the iterates are bounded, then they have finite length \cite[Theorem 3.5]{jia2019inexact} under the same parameter settings. 

Third, we discuss convergence when the initial points are chosen outside a zero measure set. The momentum method is known to converge to a local minimizer for almost every initial point under several conditions. First, $f$ should be coercive, twice differentiable, and should satisfy the Kurdyka-\L{}ojasiewicz inequality. Second, the Hessian of $f$ should have a negative eigenvalue at all critical points of $f$ that are not local minimizers. Third, the parameters of the momentum method \eqref{eq:momentum} should either satisfy $\alpha\in(0,2(1-\beta)/L)$, $\beta\in(0,1)$ and $\gamma=0$ \cite[Lemma 2]{sun2019heavy}, or $\alpha\in(0,4/L)$, $\beta\in(\max\{0,-1+\alpha L/2\},1)$ and $\gamma=0$ \cite[Theorem 3]{o2019behavior}.

Our contributions are as follows. We consider objective functions $f:\mathbb{R}^n\rightarrow\mathbb{R}$ that are semialgebraic and differentiable with locally Lipschitz gradients. The generalization to arbitrary o-minimal structures on the real field \cite{van1998tame} is immediate and omitted for the sake of brevity. We show that the length of the iterates generated by the momentum method is upper bounded by an expression depending on the objective function variation, the step size, and a desingularizing function. This length formula enables us to 
show that global Lipschitz continuity of the gradient and the global growth condition are superfluous when establishing local convergence. It also enables us to establish global convergence under the assumption that the continuous-time gradient trajectories of $f$ are bounded, which is satisfied by problems \eqref{eq:pb1}, \eqref{eq:pb2}, and \eqref{eq:pb3}, as discussed in \cite{josz2023certifying,josz2023global} (provided RIP \cite{candes2005decoding} holds in \eqref{eq:pb2}). As a result, we bypass the need for coercivity and globally Lipschitz gradients. Finally, the length formula enables us to guarantee convergence to local minimizers almost surely, under second-order differentiability and the strict saddle property \cite{lee2016,pemantle1990nonconvergence}.

This paper is organized as follows. Section \ref{sec:Convergence results} contains the statement of the length formula and the ensuing convergence results. Section \ref{sec:Proof of the length formula} contains the proof of the length formula. Section \ref{sec:Conclusion} contains the conclusion. Finally, several proofs are deferred to the Appendix for ease of readability.

\section{Convergence results}
\label{sec:Convergence results} We begin by recalling standard notations and definitions. Let $\mathbb{N} := \{0,1,2,\hdots\}$ and let $\|\cdot\|$ be the induced norm of an inner product $\langle \cdot, \cdot\rangle$ on $\mathbb{R}^n$. Let $B(a,r)$ and $\mathring{B}(a,r)$ respectively denote the closed and open balls of center $a\in \mathbb{R}^n$ and radius $r \geqslant 0$. 

Given $k\in \mathbb{N}$, $A\subset \mathbb{R}^n$ and $B \subset \mathbb{R}^m$, let $C^k(A,B)$ be the set of continuous functions $f:A\rightarrow B$ such that, if $k \geqslant 1$ then $f$ is $k$ times continuously differentiable on the interior of $A$. Let $C^{1,1}_{\rm loc}(A,B)$ denote the set of functions in $C^1(A,B)$ whose first-order derivative is locally Lipschitz continuous on the interior of $A$. When $B=\mathbb{R}$, $C^k(A,B)$ and $C^{1,1}_{\rm loc}(A,B)$ are abbreviated as $C^k(A)$ and $C^{1,1}_{\rm loc}(A)$ respectively. Let $\partial f:\mathbb{R}^n\rightrightarrows\mathbb{R}^n$ denote the Clarke subdifferential \cite[Chapter 2]{clarke1990} of a locally Lipschitz continuous function $f:\mathbb{R}^n\rightarrow \mathbb{R}$. Given a locally Lipschitz function $f:\mathbb{R}^n\rightarrow\mathbb{R}$, a real number $v$ is a critical value of $f$ in $S$ if there exists $x \in S$ such that $v = f(x)$ and $0\in \partial f(x)$. A real number $v$ is a critical value of $f$ if it is a critical value of $f$ in $\mathbb{R}^n$.

A subset $S$ of $\mathbb{R}^n$ is semialgebraic \cite{bochnak2013real,pham2016genericity} if it is a finite union of sets of the form $\{ x \in \mathbb{R}^n: p_i(x) = 0, ~ i = 1,\hdots,k ; ~  p_i(x) > 0, ~ i = k+1,\hdots,m \}$ where $p_1,\hdots,p_m$ are polynomials defined from $\mathbb{R}^n$ to $\mathbb{R}$. A function $f:\mathbb{R}^n\rightarrow \mathbb{R}$ is semialgebraic if its graph, that is to say $\{ (x,t) \in \mathbb{R}^{n+1}: f(x)=t \}$, is a semialgebraic set. We recall the following useful result.

\begin{lemma}[semialgebraic Morse-Sard theorem \cite{bolte2007clarke}]
\label{lemma:morse_sard}
    Let $f:\mathbb{R}^n\rightarrow \mathbb{R} $ be locally Lipschitz and semialgebraic. Then $f$ has finitely many critical values.
\end{lemma}

Given a subset of $S$ of $\mathbb{R}^n$, let $\mathring{S}$ and $\overline{S}$ denote the interior and closure of $S$ in $\mathbb{R}^n$ respectively. A function $\psi:S \rightarrow S$ is a homeomorphism if it is a continuous bijection and the inverse function $\psi^{-1}$ is continuous. $\psi:S \rightarrow S$ is a diffeomorphism if $\mathring{S} \neq \emptyset$, $\psi$ is a homeomorphism, and both $\psi$ and $\psi^{-1}$ are continuously differentiable on $\mathring{S}$.

We are now ready to state the key lemma upon which rest all the convergence results in this manuscript. It is entirely new to the best of our knowledge.

\begin{lemma}[Length formula]
\label{lemma:length_formula}
Let $f\in C^{1,1}_{\rm loc}(\mathbb{R}^n)$ be semialgebraic, $X\subset\mathbb{R}^n$ be bounded,  $\beta\in (-1,1)$, $\gamma\in \mathbb{R}$, and $\delta\geqslant 0$. There exist $\bar{\alpha},\eta,\kappa>0$ and a diffeomorphism $\psi:[0,\infty) \rightarrow [0,\infty)$ such that, for all $K\in\mathbb{N}$, $\alpha\in(0,\bar{\alpha}]$, and sequences $(x_k)_{k\in\{-1\}\cup\mathbb{N}}$ generated by the momentum method \eqref{eq:momentum} for which $x_{-1},\hdots, x_K \in X$ and $\|x_0-x_{-1}\|\leqslant \delta\alpha$, we have
\begin{equation*}
   \boxed{\sum_{k=0}^{K}\|x_{k+1}-x_k\| ~\leqslant~ \psi(f(x_0)-f(x_{K})+\eta\alpha) + \kappa\alpha.}
\end{equation*}
\end{lemma}

The significance of this formula is that it relates the length of the iterates with the objective function variation, in spite of the fact that the objective function values generated by the momentum method are notoriously nonmonotonic. The proof of Lemma \ref{lemma:length_formula} is quite involved so we defer it to Section \ref{sec:Proof of the length formula}. There, the reader will learn that one can actually take $\psi$ to be a desingularizing function of $f$ on $X$ (see Proposition \ref{prop:UKL}). 

We next provide some intuition on the constants in the length formula. They are constructed explicitly using the regularity of the objective function and the momentum parameters. Both $\eta$ and $\kappa$ increase with the number of critical values of $f$ in $X$ and the initial velocity $\delta$ in the momentum method. The constant $\eta$ increases with the minimal Lipschitz constant of $\nabla f$ over a certain bounded set and with the magnitude of the momentum $|\beta|$. The constant $\kappa$ increases with the minimal Lipschitz constant of $f$ over a certain bounded set.

Before we proceed, we state the following simple fact regarding the gradient of the objective function at iterates produced by the momentum method. Its proof is in Appendix \ref{sec:Proof of Fact fact:gradupper}.
\begin{fact}\label{fact:gradupper}
Let $f\in C^{1,1}_{\rm loc}(\mathbb{R}^n)$, $X\subset\mathbb{R}^n$ be bounded, and $\beta,\gamma\in \mathbb{R}$. For all $\alpha>0$, there exists $b_\alpha>0$ such that for all $K \in \mathbb{N}$, if $x_{-1},\hdots, x_{K+1} \in X$ are iterates of the momentum method \eqref{eq:momentum}, then $\| \nabla f(x_k)\| \leqslant b_\alpha \| z_{k+1}-z_k\|$ for $k=0,\hdots,K$ where $z_k:=(x_k,x_{k-1})\in\mathbb{R}^{2n}$. If $M>0$ is a Lipschitz constant of $\nabla f$ on $S+\max\{|\beta|,|\gamma|\}(S-S)$ where $S$ is the convex hull of $X$, then one may take $b_\alpha :=\sqrt{2}\max\left\{1/\alpha,|\beta|/\alpha+M|\gamma|\right\}$.
\end{fact}

We are now ready to state our first convergence result.

\begin{theorem}[Local convergence]
\label{thm:local_convergence}
Let $f\in C^{1,1}_{\rm loc}(\mathbb{R}^n)$ be semialgebraic, $\beta\in (-1,1)$, $\gamma\in \mathbb{R}$, $\delta\geqslant 0$, and $x^* \in \mathbb{R}^n$ be a local minimizer of $f$. For all $\epsilon>0$, there exist $\bar{\alpha},\xi>0$ such that for all $\alpha\in(0,\bar{\alpha}]$ and for all sequence $(x_k)_{k\in\{-1\}\cup\mathbb{N}}$ generated by the momentum method \eqref{eq:momentum} for which $\|x_0-x_{-1}\|\leqslant \delta\alpha$ and $x_0 \in B(x^*,\xi)$, $(x_k)_{k\in\{-1\}\cup\mathbb{N}}$ converges to a local minimizer of $f$ in $B(x^*,\epsilon)$.
\end{theorem}
\begin{proof}
Without loss of generality, we may assume that $f(x)\geqslant f(x^*)$ for all $x \in B(x^*,2\epsilon)$. Since $f$ is continuous and has finitely many critical values by the semialgebraic Morse-Sard theorem (Lemma \ref{lemma:morse_sard}), we may also assume that $f(x^*)$ is the unique critical value in $B(x^*,2\epsilon)$. By Lemma \ref{lemma:length_formula}, there exist $\widetilde{\alpha},\eta>0$, $\kappa>\delta$, and a diffeomorphism $\psi:[0,\infty) \rightarrow [0,\infty)$ such that, for all $K\in\mathbb{N}$, $\alpha\in(0,\widetilde{\alpha}]$, and sequences $(x_k)_{k\in\{-1\}\cup\mathbb{N}}$ generated by the momentum method \eqref{eq:momentum} for which $x_{-1},\hdots, x_K \in B(x^*,\epsilon)$ and $\|x_0-x_{-1}\|\leqslant \delta\alpha$, we have
\begin{equation}
\label{eq:loc_length_discrete_K}
   \sum_{k=0}^{K}\|x_{k+1}-x_k\| ~\leqslant~ \psi\left(f(x_0)-f(x_{K})+\eta\alpha\right) + \kappa\alpha.
\end{equation}
By the continuity of $f$, there exists $\xi \in(0, \epsilon/2]$ such that
\begin{equation}
\label{eq:loc_xi}
    f(x) - f(x^*) \leqslant \frac{1}{2}~\psi^{-1}\left(\frac{\epsilon}{6}\right)~~,~~\forall x \in B(x^*,\xi).
\end{equation}
Let
\begin{equation}
\label{eq:bar_loc}
    \bar{\alpha} := \min\left\{\widetilde{\alpha}, \frac{\epsilon}{3\kappa}, \frac{1}{3\eta}\psi^{-1}\left(\frac{\epsilon}{6}\right) \right\}.
\end{equation}
We fix any $\alpha \in (0,\bar{\alpha}]$ from now on. Let $(x_k)_{k\in\{-1\}\cup\mathbb{N}}$ be a sequence generated by the momentum method \eqref{eq:momentum} for which $\|x_0-x_{-1}\|\leqslant \delta\alpha$ and $x_0 \in B(x^*,\xi)$. If $K := \inf\{k \in \mathbb{N} : x_k \notin B(x^*,\epsilon)\}<\infty$, then  
\begin{subequations}
    \begin{align}
        \psi^{-1}\left(\frac{\epsilon}{6}\right) & =  \psi^{-1}\left(\frac{1}{2} \epsilon - \kappa\frac{\epsilon}{3\kappa} \right) \label{eq:loc_decrease_discrete_a} \\
        & \leqslant \psi^{-1}\left((\epsilon-\xi) - \kappa\bar{\alpha} \right)  \label{eq:loc_decrease_discrete_b}\\
        & \leqslant \psi^{-1}\left( (\|x_{K}-x^*\|-\|x_{0}-x^*\|) - \kappa\bar{\alpha} \right) \label{eq:loc_decrease_discrete_c}\\
        & \leqslant \psi^{-1}\left(\|x_{K}-x_{0}\| - \kappa\alpha \right) \label{eq:loc_decrease_discrete_d}\\
        & \leqslant \psi^{-1}\left( \sum_{k=0}^{K-1} \|x_{k+1}-x_k\| - \kappa\alpha \right) \label{eq:loc_decrease_discrete_e}\\
        & \leqslant f(x_{0}) - f(x_{K-1}) +  \eta\alpha \label{eq:loc_decrease_discrete_f}\\
        & \leqslant f(x_0) - f(x^*) +\eta\bar{\alpha} \label{eq:loc_decrease_discrete_g}\\
        & \leqslant \frac{1}{2} \psi^{-1}\left(\frac{\epsilon}{6}\right) + \frac{1}{3}\psi^{-1}\left(\frac{\epsilon}{6}\right)\label{eq:loc_decrease_discrete_h}\\
        &< \psi^{-1}\left(\frac{\epsilon}{6}\right).\label{eq:loc_decrease_discrete_i}
    \end{align}
\end{subequations}
As $\psi^{-1}\left(\epsilon/6\right)>0$, a contradiction occurs and thus $K = \infty$. Above, the arguments of $\psi^{-1}$ in \eqref{eq:loc_decrease_discrete_a} are equal. \eqref{eq:loc_decrease_discrete_b} through \eqref{eq:loc_decrease_discrete_e} rely on the fact that $\psi^{-1}$ is an increasing function. \eqref{eq:loc_decrease_discrete_b} is due to $\xi \leqslant\epsilon/2$ and $\bar{\alpha}\leqslant \epsilon/(3\kappa)$ by the definition of $\bar{\alpha}$ in \eqref{eq:bar_loc}. \eqref{eq:loc_decrease_discrete_c} holds because $x_K \notin B(x^*,\epsilon)$ and $x_{0} \in B(x^*,\xi)$. \eqref{eq:loc_decrease_discrete_d} and \eqref{eq:loc_decrease_discrete_e} are consequences of the triangular inequality. \eqref{eq:loc_decrease_discrete_f} is due to the length formula \eqref{eq:loc_length_discrete_K} and the fact that $x_0,\hdots,x_{K-1} \in B(x^*,\epsilon)$ and $x_{-1} \in B(x_0,\delta\alpha) \subset B(x_0,\delta\bar{\alpha}) \subset B(x^*,\xi + \delta\epsilon/(3\kappa))\subset B(x^*,\epsilon/2 + \delta\epsilon/(3\delta)) \subset B(x^*,\epsilon)$ by the definition of $\bar{\alpha}$ in \eqref{eq:bar_loc}. \eqref{eq:loc_decrease_discrete_g} is due to $f(B(x^*,\epsilon)) \subset [f(x^*),\infty)$. Finally, \eqref{eq:loc_decrease_discrete_h} is due to $x_0 \in B(x^*,\xi)$, the choice of $\xi$ as in \eqref{eq:loc_xi}, and $\bar{\alpha} \leqslant \psi^{-1}(\epsilon/6)/(3\eta)$ by definition of $\bar{\alpha}$ in \eqref{eq:bar_loc}.

We have shown that $(x_k)_{k\in\{-1\}\cup\mathbb{N}} \subset B(x^*,\epsilon)$. By the length formula \eqref{eq:loc_length_discrete_K}, we have
\begin{equation*}
    \sum_{k = 0}^\infty \|x_{k+1} - x_k\| \leqslant \psi\left( \max_{B(x^*,\xi)} f - \min_{B(x^*,\epsilon)} f +\eta\alpha\right) +\kappa\alpha.
\end{equation*}
Thus the sequence admits a limit $x^\sharp \in B(x^*,\epsilon)$. Combining with Fact \ref{fact:gradupper}, $x^\sharp$ must be a critical point of $f$. As $f(x^*)$ is the unique critical value in $B(x^*,2\epsilon)$, $f(x^\sharp) = f(x^*) \leqslant f(x)$ for all $x \in B(x^\sharp,\epsilon) \subset B(x^*,2\epsilon)$.
\end{proof}

Note that once local convergence is established, \cite[Theorem 3.3]{ochs2018local} can be applied in order to obtain local convergence rates of the iterates. Indeed, the reader will be able to check later that the assumptions \cite[(H1), (H2), (H3)]{ochs2018local} then hold (using Lemma \ref{lemma:lyapunov} and Lemma \ref{lem:gradupper} below). The rates also rely on the fact that one can take the diffeomorphism $\psi$ to be of the form $\psi(t) = ct^\theta$ where $c>0$ and $\theta \in (0,1]$ for semialgebraic functions (using Proposition \ref{prop:lyapunov_uniformKL} below).

In order to go from local convergence to global convergence, we make an assumption regarding the continuous-time gradient trajectories of the objective function. Given $f \in C^1(\mathbb{R}^n)$, we refer to maximal solutions to $x'(t) = - \nabla f(x(t))$ for all $t \in (0,T)$ where $T \in (0, \infty]$ as continuous gradient trajectories (see \cite[Chapter 17]{attouch2014variational}, \cite{garrigos2015descent}, and \cite[Section 3]{josz2023global} for background and properties). We say that a continuous gradient trajectory $x: [0,T) \rightarrow \mathbb{R}^n$ is bounded if there exists $c>0$ such that $\|x(t)\| \leqslant c$ for all $t\in [0,T)$. This assumption enables us to use a generalized version of a tracking lemma recently proposed by Kovachki and Stuart \cite[Theorem 2]{kovachki2021continuous}. 

Their result states that the momentum method tracks continuous gradient trajectories up to any given time for all sufficient small constant sizes. In Lemma \ref{lemma:tracking} below, we relax their strong regularity assumptions which require the objective function to be thrice differentiable with bounded derivatives. We instead only require it to be differentiable with a locally Lipschitz gradient. In order to do so, we redefine the key quantity $M_k$ in the proof of \cite[Theorem 2]{kovachki2021continuous}, which regulates the tracking error and depends on the Hessian of the objective function, so that it depends only on the gradient. In contrast to \cite[Theorem 2]{kovachki2021continuous}, we also make the tracking uniform with respect to the initial point in a bounded set. Choosing an upper bound on the step size in order to achieve this requires some care and cannot be deduced from the proof of \cite[Theorem 2]{kovachki2021continuous}. Below, we use the notation $\lfloor t \rfloor$ to denote the floor of a real number $t$ which is the unique integer such that $\lfloor t \rfloor \leqslant t < \lfloor t \rfloor + 1$.

\begin{lemma}[Tracking]
\label{lemma:tracking}
    Let $f\in C^{1,1}_{\rm loc}(\mathbb{R}^n)$ be a lower bounded function, $\beta\in(-1,1)$, $\gamma \in \mathbb{R}$, $\delta \geqslant 0$. For any bounded set $X_0 \subset \mathbb{R}^n$ and $\epsilon,T>0$, there exists $\bar{\alpha}>0$ such that for all $\alpha \in (0, \bar{\alpha}]$ and for any sequence $x_{-1},x_0, x_1, \ldots \in \mathbb{R}^n$ generated by the momentum method \eqref{eq:momentum} for which $x_0\in X_0$ and $\|x_0-x_{-1}\|\leqslant \delta \alpha$, there exists $x(\cdot) \in  C^{1}([0,T],\mathbb{R}^n)$ such that
    \begin{equation*}
                x'(t) = -\frac{1}{1 - \beta} \nabla f(x(t)),~~~ \forall
                t \in (0,T), ~~~ x(0) \in X_0, 
    \end{equation*}
    for which $\|x_k - x(k\alpha) \| \leqslant \epsilon$ for $k = 0, \ldots, \lfloor T/\alpha \rfloor$.
\end{lemma}

The proof of Lemma \ref{lemma:tracking} is deferred to Appendix \ref{sec:Proof of Lemma lemma:tracking}. We will also use the following simple fact in order to control the length of a single step of the momentum method as a function of the step size. Its short proof can be found in Appendix \ref{sec:Proof of Fact fact:O(alpha)}. 

\begin{fact}
\label{fact:O(alpha)}
Let $f:\mathbb{R}^n\rightarrow\mathbb{R}$ be differentiable and Lipschitz continuous on $X \subset \mathbb{R}^n$, $\beta\in(-1,1)$, $\gamma\in\mathbb{R}$, and $\delta_0 \geqslant 0$. There exists $\delta_1 \geqslant 0$ such that for all $\alpha>0$, $K \in \mathbb{N}$, and sequence $(x_{k})_{k\in \{-1\}\cup\mathbb{N}}$ generated by the momentum method \eqref{eq:momentum} for which $x_{-1},\hdots,x_{K-1}\in X$ and $\|x_0-x_{-1}\| \leqslant \delta_0\alpha$, we have 
\begin{align*}
    \|x_k-x_{k-1}\| &\leqslant \delta_1\alpha,\quad k = 0,\ldots, K,  \\
    \|z_k-z_{k-1}\| &\leqslant \sqrt{2}\delta_1\alpha,\quad k = 1,\ldots, K,
\end{align*}
where $z_k:=(x_k,x_{k-1})\in\mathbb{R}^{2n}$. If $L>0$ is a Lipschitz constant of $\bar{\beta}f$ on $S+\gamma(S-S)$ where $S$ is the convex hull of $X$ and $\bar{\beta}:= (1-\beta)^{-1}$, then we may take $\delta_1:=\delta_0+L$. 
\end{fact}

Finally, we will use the following result.

\begin{lemma}[\hspace{-.1mm}{\cite[Lemma 1]{josz2023global}}]
\label{lemma:uniform_boundedness}
Let $f \in C^1(\mathbb{R}^n)$ be a semialgebraic function with bounded continuous gradient trajectories. If $X_0 \subset \mathbb{R}^n$ is bounded, then $\sigma(X_0)<\infty$ where
\begin{align*}
    \sigma(X_0) := & \sup\limits_{x \in C^1(\mathbb{R}_+,\mathbb{R}^n)} ~~ \int_0^{\infty} \|x'(t)\|dt \\
  & ~~ \mathrm{subject~to} ~~~ 
\left\{ 
\begin{array}{l}
x'(t) = - \nabla f(x(t)), ~\forall t > 0,\\[3mm] x(0) \in X_0.
\end{array}
\right.
\end{align*}
\end{lemma}

We are now ready to state our second convergence result. It shows that, similar to the gradient method \cite[Theorem 1]{josz2023global}, the momentum method is endowed with global convergence if continuous gradient trajectories are bounded. In the gradient method, one considers the supremum of the lengths of all discrete gradient trajectories over all possible initial points in a bounded set and over all possible step sizes \cite[Equation (28)]{josz2023global}. This enables one to reason by induction on the initial set and the upper bound on the step sizes. 

When dealing with momentum, one needs to additionally consider an upper bound on the initial velocity $\|x_0-x_{-1}\|/\alpha$ between two initial points in the inductive reasoning. Fact \ref{fact:O(alpha)} guarantees that the velocity $\|x_k-x_{k-1}\|/\alpha$ remains bounded within each induction step. This enables one to reinitialize the momentum method after an arbitrary large number of iterations. Note that the length formula in Lemma \ref{lemma:length_formula} admits an error term $\eta \alpha$ that is not present in the gradient method \cite[Proposition 9]{josz2023global}. This requires additional care.

\begin{theorem}[Global convergence]
\label{thm:convergence}
Let $f\in C^{1,1}_{\rm loc}(\mathbb{R}^n)$ be semialgebraic with bounded continuous gradient trajectories. Let $\beta \in (-1,1)$, $\gamma \in \mathbb{R}$, $\delta \geqslant 0$ and $X_0$ be a bounded subset of $\mathbb{R}^n$. There exist $\bar{\alpha},c>0$ such that for all $\alpha \in (0,\bar{\alpha}]$, there exists $c_\alpha>0$ such that any sequence $x_{-1},x_0,x_1,\hdots \in \mathbb{R}^n$ generated by the momentum method \eqref{eq:momentum} that satisfies $x_0 \in X_0$ and $\|x_0-x_{-1}\| \leqslant \delta \alpha$ obeys
\begin{equation}
\label{eq:rate}
        \sum_{i=0}^{\infty} \|x_{i+1}-x_i\| \leqslant c ~~~ \text{and} ~~~ \min_{i=0,\ldots,k}\|\nabla f(x_i)\| \leqslant \frac{c_\alpha}{k+1}, ~ \forall k\in \mathbb{N}.
\end{equation}
\end{theorem}

\begin{proof}
Let $\beta \in (-1,1)$, $\gamma \in \mathbb{R}$, $\delta_0\geqslant 0$, and $X_0$ be a bounded subset of $\mathbb{R}^n$. Without loss of generality, we may assume that $X_0 \neq \emptyset$. We will show that there exists $\bar{\alpha}>0$ such that $\sigma(X_0,\bar{\alpha},\delta_0)<\infty$ where
\begin{subequations}
\label{eq:sup_discrete_X_0}
\begin{align}
    & \sigma(X_0,\bar{\alpha},\delta_0) := \sup\limits_{\tiny \begin{array}{c}x \in (\mathbb{R}^n)^{\mathbb{N}}\\ \alpha \in (0,\bar{\alpha}] \end{array}} ~~ \sum\limits_{k=0}^{\infty} \|x_{k+1}-x_k\| \\
    & \mathrm{s.t.} ~~
\left\{ 
\begin{array}{l}
x_{k+1} = x_k + \beta(x_k-x_{k-1}) -\alpha \nabla f (x_k + \gamma(x_k-x_{k-1})),~\forall k \in \mathbb{N},\\[2mm]   x_0 \in X_0, ~ \|x_0-x_{-1}\| \leqslant \delta_0 \alpha.
\end{array}
\right.
\end{align}
\end{subequations}

Letting $c := \sigma(X_0,\bar{\alpha},\delta_0)$, the convergence rate is easily deduced. Indeed, for any feasible point $((x_k)_{k \in \{-1\} \cup \mathbb{N}},\alpha)$ of \eqref{eq:sup_discrete_X_0}, we have $x_{-1},x_0,x_1, \hdots \in B(X_0,\max\{c,\delta_0 \alpha\}):= X_0 + B(0,\max\{c,\delta_0 \alpha\})$ and thus $\|\nabla f(x_k)\| \leqslant b_\alpha \|z_{k+1}-z_{k}\|$ for all $k \in \mathbb{N}$ for some constant $b_\alpha>0$ by Lemma \ref{lem:gradupper}, where $z_k:= (x_k,x_{k-1})$. Hence
\begin{align*}
    \sum_{k=0}^\infty\|\nabla f(x_k)\| & \leqslant \sum_{k=0}^\infty b_\alpha \|z_{k+1} - z_k\| \\
    & \leqslant b_\alpha \|x_0-x_{-1}\|+2b_\alpha\sum_{k=0}^\infty\|x_{k+1}-x_k\| \\
    & \leqslant b_\alpha (\delta_0 \alpha + 2 c) =: c_\alpha
\end{align*}
and
\begin{equation*}
    \min_{i=0,\ldots,k}\|\nabla f(x_i)\| \leqslant \frac{1}{k+1}\sum_{i=0}^k\|\nabla f(x_i)\| \leqslant \frac{c_\alpha}{k+1}.
\end{equation*}

Let $\Phi:\mathbb{R}_+ \times \mathbb{R}^n \rightarrow \mathbb{R}^n$ be the continuous gradient flow of $f$ defined for all $(t,x_0) \in \mathbb{R}_+ \times \mathbb{R}^n$ by $\Phi(t,x_0) := x(t)$ where $x(\cdot)$ is the unique continuous gradient trajectory of $f$ initialized at $x_0$. Uniqueness follows from the Picard–Lindel{\"o}f theorem \cite[Theorem 3.1 p. 12]{coddington1955theory}. Let $\Phi_0 := \Phi(\mathbb{R}_+,X_0)$ and let $C$ be the set of critical points of $f$ in $\overline{\Phi}_0$. $C$ is compact by Lemma \ref{lemma:uniform_boundedness} and \cite[2.1.5 Proposition p. 29]{clarke1990}. Thus there exists $\epsilon >0$ such that either $X_0 \subset C$ or $X_0\setminus \mathring{B}(C,\epsilon/6) \neq \emptyset$ where $\mathring{B}(C,\epsilon/6) := C + \mathring{B}(0,\epsilon/6)$. 

Indeed, either $X_0 \subset C$ or there exists $x \in X_0 \cap C^c$ where the complement $C^c$ of $C$ is open since $C$ is closed. Thus there exists $\epsilon >0$ such that $B(x,\epsilon/6) \subset C^c$. Thus $x \not\in \mathring{B}(C,\epsilon/6)$ (otherwise there exists $x' \in C$ such that $\|x-x'\|< \epsilon/6$, i.e., $C \ni x' \in B(x,\epsilon/6) \subset C^c$) and $x \in  X_0 \setminus \mathring{B}(C,\epsilon/6)$.

By Fact \ref{fact:O(alpha)}, there exists $\delta_1 > \delta_0$ such that for all $\alpha>0$, $K \in \mathbb{N}$, and sequence $(x_{k})_{k\in \{-1\}\cup\mathbb{N}}$ generated by the momentum method \eqref{eq:momentum} for which $x_{-1},\hdots,x_{K-1}\in B(\overline{\Phi}_0,\epsilon):= \overline{\Phi}_0 + B(0,\epsilon)$ and $\|x_0-x_{-1}\| \leqslant \delta_0\alpha$, we have $\|x_k - x_{k-1}\| \leqslant \delta_1 \alpha$ for $k = 0, \ldots, K$. By Lemma \ref{lemma:length_formula}, there exist $\widetilde{\alpha},\eta>0$, $\kappa\geqslant\delta_1$, and 
a diffeomorphism $\psi:[0,\infty)\rightarrow[0,\infty)$ such that for all $K\in\mathbb{N}\setminus \{0\}$, $\alpha\in(0,\widetilde{\alpha}]$, and sequence $(x_{k})_{k\in \{-1\}\cup\mathbb{N}}$ generated by the momentum method \eqref{eq:momentum} for which $x_{-1},\hdots,x_{K-1}\in B(\overline{\Phi}_0,\epsilon)$ and $\|x_0-x_{-1}\| \leqslant \delta_1\alpha$, we have
\begin{equation}
\label{eq:length_discrete_K}
    \sum_{k=0}^{K-1}\|x_{k+1}-x_k\| ~\leqslant~ \psi\left(f(x_0)-f(x_{K-1})+ \eta\alpha\right) + \kappa\alpha.
\end{equation}

Since $f$ is continuous, there exists $\xi \in (0,\epsilon/2)$ such that
\begin{equation}
    \label{eq:X_1'}
     f(x) - \max_{C} f \leqslant \frac{1}{4} \psi^{-1}\left(\frac{\epsilon}{3}\right),~~~ \forall x\in B(C,\xi).
\end{equation}
Let $L>0$ be a Lipschitz constant of $\bar{\beta}f$ on the convex hull of $B(\overline{\Phi}_0,\epsilon)$ and let
\begin{equation}
\label{eq:alpha_hat}
    \hat{\alpha} := \min\left\{\widetilde{\alpha},\frac{\xi}{3L} ,\frac{\epsilon}{6\kappa},\frac{\psi^{-1}(\epsilon/3)}{4\eta}\right\}>0
\end{equation}
where $\bar{\beta}:= 1/(1-\beta)$. If $X_0\subset C$, then let $\bar{\alpha} := \hat{\alpha}$ and $k^* := 0$. Otherwise $X_0\setminus \mathring{B}(C,\epsilon/6) \neq \emptyset$. Since $\bar{\beta}\nabla f$ is continuous, its norm attains its infimum $\nu$ on the non-empty compact set $\overline{\Phi}_0 \setminus \mathring{B}(C,\xi/3)$.  
It is non-empty because $\overline{\Phi}_0 \setminus \mathring{B}(C,\xi/3) \supset X_0\setminus \mathring{B}(C,\epsilon/6) \neq \emptyset$ and $\xi < \epsilon/2$. If $\nu=0$, then there exists $x^* \in \overline{\Phi}_0 \setminus \mathring{B}(C,\xi/3)$ such that $\|\nabla f(x^*)\| = 0$. Then $x^*\in C \setminus \mathring{B}(C,\xi/3)$, which is a contradiction. We thus have $\nu>0$. Hence we may define $T := 2\sigma(X_0)/\nu$ where
\begin{align*}
    \sigma(X_0) = & \sup\limits_{x \in  C^1(\mathbb{R}_+,\mathbb{R}^n)} ~~ \int_0^{\infty} \|x'(t)\|dt \\
  & ~~ \mathrm{subject~to} ~~~ 
\left\{ 
\begin{array}{l}
x'(t) = - \bar{\beta}\nabla f(x(t)), ~\forall t > 0,\\[3mm] x(0) \in X_0,
\end{array}
\right.
\end{align*}
is finite by Lemma \ref{lemma:uniform_boundedness}. The factor $\bar{\beta}>0$ does not change the optimal value because $x(\cdot)$ is a feasible point of the above problem if and only if $x(\cdot/\bar{\beta})$ is a feasible point of the problem in Lemma 2.7 and $\int_0^\infty \|x'(t/\bar{\beta})/\bar{\beta}\|dt = \int_0^\infty \|x'(t)\|dt$. Note that $\sigma(X_0)>0$ and thus $T>0$ because $X_0\not\subset C$. In addition, since $f$ is semialgebraic and has bounded continuous gradient trajectories, it is lower bounded by its smallest critical value\footnote{Indeed, assume to the contrary that there exists $x_0 \in \mathbb{R}^n$ such that $f(x_0)$ is less than the smallest critical value of $f$. The continuous gradient trajectory initialized at $x_0$ converges to a critical point $x^*$ since it is bounded. This limit satisfies $f(x_0) \geqslant f(x^*)$, yielding a contradiction.}. By Lemma \ref{lemma:tracking}, there exists $\bar{\alpha} \in (0,\hat{\alpha}]$ such that for any feasible point $((x_k)_{k\in \{-1\}\cup \mathbb{N}} , \alpha)$ of \eqref{eq:sup_discrete_X_0},
there exists an absolutely continuous function $x:[0,T]\rightarrow\mathbb{R}^n$ such that
\begin{equation}
    x'(t) = -\bar{\beta} \nabla f(x(t)),~~~ \forall
    t \in (0,T), ~~~ x(0) \in X_0, 
\end{equation}
for which $\|x_k - x(k\alpha) \| \leqslant \xi/3$ for $k = 0, \ldots, \lfloor T/\alpha \rfloor$. Now suppose that $\|x'(t)\| \geqslant 2\sigma(X_0)/T$ for all $t\in (0,T)$. Then we obtain the following contradiction
\begin{equation*}
     \sigma(X_0) < T \frac{2\sigma(X_0)}{T} \leqslant \int_0^{T} \|x'(t)\|dt \leqslant \int_0^{\infty} \|x'(t)\|dt \leqslant \sigma(X_0).
\end{equation*}
Hence, there exists $t^* \in (0,T)$ such that $\|x'(t^*)\| = \|\bar{\beta}\nabla f(x(t^*))\| < 2\sigma(X_0)/T = 2\sigma(X_0)/(2\sigma(X_0)/\nu) = \nu$. Since  $x(t^*) \in \overline{\Phi}_0$ and the infimum of the norm of $\bar{\beta}\nabla f$ on $\overline{\Phi}_0 \setminus \mathring{B}(C,\xi/3)$ is equal to $\nu$, it must be that $x(t^*) \in \mathring{B}(C,\xi/3)$. Hence there exists $x^* \in C$ such that $\|x(t^*)-x^*\|\leqslant \xi/3$. 

Since $\alpha \leqslant \hat{\alpha} \leqslant \xi/(3L)$, there exists $k^* \in \mathbb{N}$ such that $t_{k^*} := k^* \alpha \in [ t^* - \xi/(3L) , t^*]$. Thus $\|x_{k^*}-x^*\| \leqslant \|x_{k^*}-x(t_{k^*})\| + \|x(t_{k^*}) - x(t^*)\| + \|x(t^*)-x^*\| \leqslant \xi/3 + L | t_{k^*}-t^* | + \xi/3 \leqslant \xi/3+\xi/3+\xi/3 = \xi$. To obtain the second inequality, we used the fact that for all $t \geqslant 0$, we have $\|x'(t)\| = \|\bar{\beta}\nabla f(x(t))\| \leqslant L$ since $x(t) \in B(\overline{\Phi}_0,\epsilon)$.

Above, we defined $\bar{\alpha} \in (0, \hat{\alpha}]$ if $X_0 \subset C$ or $X_0 \not\subset C$ (in which case $X_0\setminus \mathring{B}(C,\epsilon/6) \neq \emptyset$). We now consider a feasible point $((x_k)_{k\in \{-1\}\cup \mathbb{N}} , \alpha)$ of \eqref{eq:sup_discrete_X_0} regardless of whether $X_0 \subset C$. Based on the above and the fact that $X_0 \neq \emptyset$, there exists $k^* \in \mathbb{N}$ such that $x_{k^*} \in B(C,\xi)$. 
If $K := \inf\{k \geqslant k^* : x_k \notin B(C,\epsilon)\}<\infty$, then
\begin{subequations}
    \begin{align}
        \psi^{-1}\left(\frac{\epsilon}{3}\right) & =  \psi^{-1}\left(\frac{1}{2} \epsilon - \kappa\frac{\epsilon}{6\kappa} \right) \label{eq:decrease_discrete_a} \\
        & \leqslant \psi^{-1}\left((\epsilon-\xi) - \kappa\hat{\alpha} \right) \label{eq:decrease_discrete_b} \\
        & \leqslant \psi^{-1}\left(( \|x_{K}-x^*\|-\|x_{k^*}-x^*\|) - \kappa\bar{\alpha} \right) \label{eq:decrease_discrete_c} \\
        & \leqslant \psi^{-1}\left(\|x_{K}-x_{k^*}\| - \kappa\alpha \right) \label{eq:decrease_discrete_d} \\
        & \leqslant \psi^{-1}\left(\sum_{k=k^*}^{K-1} \|x_{k+1}-x_k\| - \kappa\alpha \right) \label{eq:decrease_discrete_e} \\
        & \leqslant f(x_{k^*}) - f(x_{K-1}) +  \eta\alpha \label{eq:decrease_discrete_f} \\
        & \leqslant \max_C f + \frac{1}{4} \psi^{-1}\left(\frac{\epsilon}{3}\right) - f(x_{K-1}) +\eta\hat{\alpha} \label{eq:decrease_discrete_g} \\
        & \leqslant \max_C f + \frac{1}{2} \psi^{-1}\left(\frac{\epsilon}{3}\right) - f(x_{K-1}) \label{eq:decrease_discrete_h}
    \end{align}
\end{subequations}
Above, the arguments of $\psi^{-1}$ in \eqref{eq:decrease_discrete_a} are equal. \eqref{eq:decrease_discrete_b} through \eqref{eq:decrease_discrete_e} rely on the fact that $\psi^{-1}$ is an increasing function. \eqref{eq:decrease_discrete_b} is due to $\xi<\epsilon/2$. \eqref{eq:decrease_discrete_c} holds because $x_K \notin B(C,\epsilon)$, $x^* \in C$, and $x_{k^*} \in B(x^*,\xi)$. \eqref{eq:decrease_discrete_d} and \eqref{eq:decrease_discrete_e} are consequences of the triangular inequality. \eqref{eq:decrease_discrete_f} is due to the length formula \eqref{eq:length_discrete_K} and the fact that $x_{k^*},\hdots,x_{K-1} \in B(C,\epsilon) \subset B(\overline{\Phi}_0,\epsilon)$. \eqref{eq:decrease_discrete_g} is due to $x_{k^*} \in B(x^*,\xi)$ and \eqref{eq:X_1'}. Finally, \eqref{eq:decrease_discrete_h} is due to $\hat{\alpha} \leqslant \psi^{-1}(\epsilon/3)/(4\eta)$ by definition of $\hat{\alpha}$ in \eqref{eq:alpha_hat}. We remark that $K\geqslant k^*+2$ since $\|x_{k^*+1}-x^*\| \leqslant \|x_{k^*+1}-x_{k^*}\| +\|x_{k^*}-x^*\| \leqslant \delta_1 \alpha+\xi \leqslant \delta_1\epsilon/(6\kappa)  +\epsilon/2 \leqslant \epsilon/6  +\epsilon/2 < \epsilon$. It also holds that $x_{-1},\hdots,x_{K-2} \in B(\overline{\Phi}_0,\epsilon)$, $x_{K-1}$ belongs to
\begin{equation*}
    X_1 ~:=~ B(C,\epsilon) ~ \bigcap ~ \left\{ x\in \mathbb{R}^n : f(x) \leqslant \max_C f - \frac{1}{2} \psi^{-1}\left(\frac{\epsilon}{3}\right) \right\},
\end{equation*}
and $\|x_{K-1} - x_{K-2}\| \leqslant \delta_1 \alpha$. Thus, by the length formula \eqref{eq:length_discrete_K} and the definition of $\sigma(\cdot,\cdot,\cdot)$ in \eqref{eq:sup_discrete_X_0} we have
\begin{align*}
    \sum_{k=0}^{\infty} \|x_{k+1}-x_k\| & = \sum_{k=0}^{K-2} \|x_{k+1}-x_k\| +  \sum_{k=K-1}^{\infty} \|x_{k+1}-x_k\|  \\
     & \leqslant \psi\left(\sup_{X_0} f - \min_{B(\overline{\Phi}_0,\epsilon)} f +  \eta \bar{\alpha} \right) + \kappa\bar{\alpha} +\max\{0, \sigma(X_1,\bar{\alpha},\delta_1)\}.
\end{align*}
Note that the inequality still holds if $K = \infty$, since in that case \eqref{eq:length_discrete_K} implies
\begin{equation*}
    \sum_{k=0}^{\infty} \|x_{k+1}-x_k\| \leqslant \psi\left(\sup_{X_0} f - \min_{B(\overline{\Phi}_0,\epsilon)} f +  \eta \bar{\alpha} \right) + \kappa\bar{\alpha}.
\end{equation*}
Hence 
\begin{equation*}
    \sigma(X_0,\bar{\alpha},\delta_0) \leqslant  \psi\left(\sup_{X_0} f - \min_{B(\overline{\Phi}_0,\epsilon)} f +  \eta \bar{\alpha} \right) + \kappa\bar{\alpha} +\max\{0, \sigma(X_1,\bar{\alpha},\delta_1)\}.
\end{equation*}

It now suffices to replace $X_0$ by $X_1$, $\delta_0$ by $\delta_1$, and repeat the entire proof. Since $f(\Phi(t,x_1)) \leqslant f(\Phi(0,x_1)) \leqslant \max_{C} f - \psi^{-1}(\epsilon/3)/2< \max_{C} f$ for all $t\geqslant 0$ and $x_1\in X_1$, the maximal critical value of $f$ in $\overline{\Phi(\mathbb{R}_+,X_1)}$ is less than the maximal critical value of $f$ in $\overline{\Phi(\mathbb{R}_+,X_0)}$. By the semialgebraic Morse-Sard theorem (Lemma \ref{lemma:morse_sard}), $f$ has finitely many critical values. Thus, it is eventually the case that one of the sets $X_0,X_1,\hdots$ is empty. In order to conclude, one simply needs to choose an upper bound on the step sizes $\bar{\alpha}'$ corresponding to $X_1$ that is less than or equal to the upper bound $\bar{\alpha}$ used for $X_0$. $\sigma(X_0,\cdot,\delta_0)$ is finite when evaluated at the last upper bound thus obtained. Indeed, the recursive formula above 
still holds if we replace $\bar{\alpha}$ by any $\alpha \in (0,\bar{\alpha}]$. 
In particular, we may take $\alpha := \bar{\alpha}'$.
\end{proof}

The inequalities in \eqref{eq:rate} imply that the iterates converge to a critical point of $f$ (i.e., a point $x^* \in \mathbb{R}^n$ such that $\nabla f(x^*)=0$). The previously known global convergence rate of the momentum method is $O(1/\sqrt{k})$ for coercive differentiable functions with a Lipschitz continuous gradient \cite[Theorem 4.14]{ochs2014ipiano} if $x_{-1}=x_0$ and $\gamma = 0$. Without the coercivity assumption, $O(1/\sqrt{k})$ is also the rate of a modified version of the momentum method which does not capture the heavy ball method and Nesterov's accelerated gradient method as special cases \cite[Corollary 1]{ghadimi2016accelerated}. 
Our third and final convergence result gives sufficient conditions for convergence to a local minimizer.

\begin{theorem}[Convergence to local minimizers]
\label{thm:local_min}
Let $f \in C^2(\mathbb{R}^n)$ be semialgebraic with bounded continuous gradient trajectories. Let $\beta \in (-1,1) \setminus \{0\}$, $\gamma \in \mathbb{R}$, and $\delta \geqslant 0$. If the Hessian of $f$ has a negative eigenvalue at all critical points of $f$ that are not local minimizers, then for any bounded subset $X_0$ of $\mathbb{R}^n$, there exists $\bar{\alpha}>0$ such that, for all $\alpha \in (0,\bar{\alpha}]$ and for almost every 
$(x_{-1},x_0) \in \mathbb{R}^n \times X_0$, 
any sequence $x_{-1},x_0,x_1,\hdots \in \mathbb{R}^n$ generated by the momentum method \eqref{eq:momentum} 
that satisfies $\|x_0-x_{-1}\| \leqslant \delta \alpha$ 
converges to a local minimizer of $f$.
\end{theorem}

In practice, if $X_0$ has positive measure and $\delta>0$, Theorem \ref{thm:local_min} means that one can generate $x_0$ uniformly at random in $X_0$ and generate $x_{-1}$ uniformly at random in the ball of radius $\delta \alpha$ centered at $x_0$ in order to guarantee convergence to a local minimizer almost surely. In contrast to the gradient method \cite[Corollary 1]{josz2023global}, we need to assume that the Hessian of $f$ has a negative eigenvalue at local maxima of $f$. Indeed, the function values are not necessarily decreasing along the iterates. While the proof of Theorem \ref{thm:local_min} crucially depends on the length bound in Theorem \ref{thm:convergence}, it mostly requires extending well-known arguments regarding the center and stable manifolds theorem \cite[Theorem III.7]{shub2013global}. For this reason, we defer its proof to Appendix \ref{sec:Proof of Theorem thm:local_min}.

\section{Proof of the length formula}
\label{sec:Proof of the length formula}

Given $\lambda>0$, consider the following Lyapunov function proposed by Zavriev and Kostyuk \cite{zavriev1993heavy}:
\begin{equation*}
    \begin{array}{cccc}
         H_\lambda: & \mathbb{R}^n \times \mathbb{R}^n & \longrightarrow & \mathbb{R}\\
        & (x,y) & \longmapsto & f(x) + \lambda \|x-y\|^2.
    \end{array}
\end{equation*}
For certain values of $\lambda$, it is known to be monotonic along the iterates if $\gamma = 0$ \cite[Lemma 1]{zavriev1993heavy} \cite[Proposition 4.7 (a)]{ochs2014ipiano} or $\beta = \gamma$ \cite[Lemma 3.1 (ii)]{wen2017linear} \cite[Lemma 3.2]{jia2019inexact}, but not for general $\beta \in (-1,1)$ and $\gamma \in \mathbb{R}$. This justifies the need for Lemma \ref{lemma:lyapunov} in which it will be convenient to rewrite the update rule of the momentum method \eqref{eq:momentum} as
\begin{subequations}
\label{eq:sepalg}
    \begin{align}
    y_k^\beta & = x_k + \beta(x_k-x_{k-1}), \\
    y_k^\gamma & = x_k + \gamma(x_k-x_{k-1}), \\
    x_{k+1} & = y_k^\beta - \alpha \nabla f(y_k^\gamma),
    \end{align}
\end{subequations}
for all $k \in \mathbb{N}$. Likewise, bounds on the norm of the gradient of the objective function and the Lyapunov function are only known if $\gamma = 0$ \cite[Theorem 4.9]{ochs2014ipiano} or $\beta = \gamma$ \cite[Equation (3.22)]{wen2017linear} \cite[Lemma 3.3]{jia2019inexact}, which calls for Lemma \ref{lem:gradupper}. 

\begin{lemma}\label{lemma:lyapunov}
Let $f\in C^{1,1}_{\rm loc}(\mathbb{R}^n)$, $X\subset\mathbb{R}^n$ be bounded, $\beta\in (-1,1)$, and $\gamma\in \mathbb{R}$. There exists $\bar{\alpha}>0$ such that for all $\alpha\in(0,\bar{\alpha}]$, there exist $\lambda^+>\lambda^->0$ such that for all $\lambda\in(\lambda^-,\lambda^+)$, there exists $c_1>0$ such that for all $K\in \mathbb{N}$, if $x_{-1},\hdots, x_{K+1} \in X$ are iterates of the momentum method \eqref{eq:momentum}, then for $k=0,\hdots,K$ we have
\begin{equation*}
     H_\lambda(x_{k+1},x_k) \leqslant H_\lambda(x_{k},x_{k-1}) - c_1(\|x_{k+1} - x_k\|^2+\|x_{k} - x_{k-1}\|^2).
\end{equation*}
If $M>0$ is a Lipschitz constant of $\nabla f$ on $S+\max\{|\beta|,|\gamma|\}(S-S)$ where $S$ is the convex hull of $X$, then one may take
\begin{subequations}
\begin{gather}
    \bar{\alpha}:=\min\left\{\frac{1}{M},\frac{1-\beta^2}{2(\beta^2+2|\beta-\gamma|)M}\right\}, \label{eq:alpha_range}\\
    \lambda^-:=\left(\frac{1}{2\alpha}+\frac{M}{2}\right)\beta^2+\frac{|\beta-\gamma|M}{2},~~~\lambda^+:=\frac{1}{2\alpha}-\frac{|\beta-\gamma|M}{2},  \label{eq:lambda_range} \\
    \text{and} ~~~ c_1 : = \min \{ \lambda - \lambda^-, \lambda^+ - \lambda\}. \label{eq:c_1}
\end{gather}
\end{subequations}
\end{lemma}

\begin{proof}
Consider $\bar{\alpha}$ as defined in \eqref{eq:alpha_range} and let $\alpha\in(0,\bar{\alpha}]$. Given $K\in \mathbb{N}$, let $x_{-1},\hdots, x_{K+1} \in X$ be iterates generated by the momentum method \eqref{eq:momentum}. A bound on the Taylor expansion of $f$ yields
\begin{subequations}
\begin{align}
    f(x_{k+1}) &\leqslant f(y_k^\beta) + \langle \nabla f(y_k^\beta), x_{k+1}-y_k^\beta\rangle + \frac{M}{2}\| x_{k+1}-y_k^\beta\|^2, \label{eq:upperTaylor} \\
    f(x_k) &\geqslant f(y_k^\beta) + \langle \nabla f(y_k^\beta), x_k-y_k^\beta\rangle - \frac{M}{2}\| x_k-y_k^\beta\|^2, \label{eq:lowerTaylor}
\end{align}
\end{subequations}
where $k \in \{0,\hdots,K\}$. Subtracting (\ref{eq:lowerTaylor}) from (\ref{eq:upperTaylor}) yields
\begin{subequations}\label{eq:decrease1}
\begin{align}
    f(x_{k+1})-f(x_k) \leqslant & \langle \nabla f(y_k^\beta),x_{k+1}-x_k\rangle + \frac{M}{2}(\| x_{k+1}-y_k^\beta\|^2+\| x_k-y_k^\beta\|^2) \\
    = & \langle \nabla f(y_k^\beta)-\nabla f(y_k^\gamma),x_{k+1}-x_k\rangle+ \langle\nabla f(y_k^\gamma),x_{k+1}-x_k\rangle \\
    &+\frac{M}{2}(\| x_{k+1}-y_k^\beta\|^2+\| x_k-y_k^\beta\|^2) \\
    = & \langle \nabla f(y_k^\beta)-\nabla f(y_k^\gamma),x_{k+1}-x_k\rangle+ \frac{1}{\alpha}\langle x_{k+1}-y_k^\beta,x_k-x_{k+1}\rangle \\
    &+\frac{M}{2}(\| x_{k+1}-y_k^\beta\|^2+\| x_k-y_k^\beta\|^2).
\end{align}
\end{subequations}
By the Cauchy-Schwarz and AM-GM inequalities, we have
\begin{subequations}\label{eq:decrease2}
    \begin{align}
        \langle \nabla f(y_k^\beta)-\nabla f(y_k^\gamma),x_{k+1}-x_k\rangle & \leqslant \| \nabla f(y_k^\beta)-\nabla f(y_k^\gamma) \| \| x_{k+1}-x_k\| \\
        &\leqslant M\| y_k^\beta- y_k^\gamma\| \| x_{k+1}-x_k\| \\
        &= M|\beta-\gamma|\| x_k-x_{k-1}\| \| x_{k+1}-x_k\| \\
        &\leqslant \frac{|\beta-\gamma|M}{2}(\| x_k-x_{k-1}\|^2+\| x_{k+1}-x_k\|^2).
    \end{align}
\end{subequations}
By the cosine rule, for any $a,b,c\in\mathbb{R}^n$, it holds that  
\begin{equation*}
    \langle a-b,c-a\rangle = \frac{1}{2}(\| b-c\|^2-\| a-b\|^2-\| c-a\|^2). 
\end{equation*}
By letting $a:=x_{k+1}$, $b:=y_k^\beta$ and $c:=x_k$, we have 
\begin{equation}\label{eq:decrease3}
     \langle x_{k+1}-y_k^\beta,x_k-x_{k+1}\rangle =\frac{1}{2}(\| y_k^\beta-x_k\|^2-\| x_{k+1}-y_k^\beta\|^2-\| x_k-x_{k+1}\|^2).  
\end{equation}
Combining (\ref{eq:decrease1}), (\ref{eq:decrease2}) and (\ref{eq:decrease3}), we find that
\begin{subequations}
\label{eq:f_decrease}
\begin{align}
    f(x_{k+1})&-f(x_k) \leqslant -\left(\frac{1}{2\alpha}-\frac{M}{2}\right)\| y_k^\beta-x_{k+1}\|^2 \\
    &- \left(\frac{1}{2\alpha}-\frac{|\beta-\gamma|M}{2}\right)\| x_{k+1}-x_k\|^2 \\
    &+ \left[\left(\frac{1}{2\alpha}+\frac{M}{2}\right)\beta^2+\frac{|\beta-\gamma|M}{2}\right]\| x_k-x_{k-1}\|^2.
\end{align}
\end{subequations}
Let $\lambda\in(\lambda^-,\lambda^+)$ where $\lambda^-$ and $\lambda^+$ are defined in \eqref{eq:lambda_range}. Note that $\lambda^-<\lambda^+$ due to the fact that $\alpha \in (0,\bar{\alpha}]$. By definition of $H_\lambda$, it readily follows that
\begin{align*}
    H_\lambda(x_{k+1},x_k)&-H_\lambda(x_k,x_{k-1}) \leqslant -\left(\frac{1}{2\alpha}-\frac{M}{2}\right)\| y_k^\beta-x_{k+1}\|^2 \\
    &- \left(\frac{1}{2\alpha}-\frac{|\beta-\gamma|M}{2}-\lambda\right)\| x_{k+1}-x_k\|^2 \\
    &- \left[\lambda-\left(\frac{1}{2\alpha}+\frac{M}{2}\right)\beta^2-\frac{|\beta-\gamma|M}{2}\right]\| x_k-x_{k-1}\|^2.
\end{align*}
The desired inequality is guaranteed by taking $c_1$ as defined in \eqref{eq:c_1}. 
\end{proof}

\begin{lemma}\label{lem:gradupper}
Let $f\in C^{1,1}_{\rm loc}(\mathbb{R}^n)$, $X\subset\mathbb{R}^n$ be bounded, and $\beta,\gamma\in \mathbb{R}$. For all $\alpha,\lambda>0$, there exist $c_2>0$ such that for all $K \in \mathbb{N}$, if $x_{-1},\hdots, x_{K+1} \in X$ are iterates of the momentum method \eqref{eq:momentum}, then
\begin{align*}
    \max\{\| \nabla H_\lambda(z_k)\|,\|\nabla H_\lambda(z_{k+1})\|\} \leqslant c_2 \| z_{k+1}-z_k\|, 
\end{align*}
for $k=0,\hdots,K$ where $z_k:=(x_k,x_{k-1})\in\mathbb{R}^{2n}$. If $M>0$ is a Lipschitz constant of $\nabla f$ on $S+\max\{|\beta|,|\gamma|\}(S-S)$ where $S$ is the convex hull of $X$, then one may take
\begin{equation}
\label{eq:c_2}
        c_2 :=\sqrt{2}\max\left\{\frac{1}{\alpha},\frac{|\beta|}{\alpha}+M(|\gamma|+1)+4\lambda\right\}.
\end{equation}
\end{lemma}
\begin{proof}
Using Fact \ref{fact:gradupper}, for $k=0,\hdots,K$ we have
\begin{align*}
    \|\nabla H_\lambda(z_k)\| &\leqslant \|\nabla f(x_k) + 2\lambda(x_k - x_{k-1})\| + \|2\lambda (x_k - x_{k-1})\| \\
    &\leqslant \|\nabla f(x_k)\| + 4\lambda \|x_k - x_{k-1}\| \\
    &\leqslant \sqrt{2}\max\{1/\alpha,|\beta|/\alpha+M|\gamma|+4\lambda\}\|z_{k+1} - z_k\|.
\end{align*}
Similarly, 
\begin{align*}
    \|\nabla H_\lambda(z_{k+1})\| &\leqslant \|\nabla f(x_{k+1}) + 2\lambda(x_{k+1} - x_k)\| + \|2\lambda (x_{k+1} - x_k)\| \\
    &\leqslant \|\nabla f(x_{k+1})\| + 4\lambda \|x_{k+1} - x_k\| \\
    &\leqslant \|\nabla f(x_k)\| + \|\nabla f(x_{k+1})-\nabla f(x_k)\| + 4\lambda \|x_{k+1} - x_k\| \\
    &\leqslant \|\nabla f(x_k)\| + (4\lambda+M) \|x_{k+1} - x_k\| \\
    &\leqslant \sqrt{2}\max\{1/\alpha,|\beta|/\alpha+M|\gamma|+4\lambda+M\}\|z_{k+1} - z_k\|.\qedhere
\end{align*}
\end{proof}

In order to proceed, we recall two results from the literature. Given $x\in \mathbb{R}^n$, consider the distance of $x$ to $S$ defined by $d(x,S) := \inf \{ \|x-y\| : y \in S \}$. Given a set-valued mapping $F:\mathbb{R}^n\rightrightarrows\mathbb{R}^m$ and $y \in \mathbb{R}^m$, let $F^{-1}(y) := \{ x \in \mathbb{R}^n : F(x) \ni y \}$. 

\begin{theorem}[Kurdyka-\L{}ojasiewicz inequality \cite{kurdyka1998gradients,bolte2007clarke}]\label{thm:KL}
    Let $f:\mathbb{R}^n\rightarrow \mathbb{R}$ be locally Lipschitz and semialgebraic. Let $X$ be a bounded subset of $\mathbb{R}^n$ and $v \in \mathbb{R}$ be a critical value of $f$ in $\overline{X}$. There exists $\rho>0$ and a strictly increasing continuous semialgebraic function $\psi:[0,\rho)\rightarrow [0,\infty)$ which belongs to $C^1((0,\rho))$ with $\psi(0) = 0$ such that
    \begin{equation}
    \label{eq:KL}
        \forall x \in X, ~~~ |f(x)-v| \in (0,\rho) ~~~ \Longrightarrow ~~~ d(0,\partial (\psi \circ |f-v|)(x)) \geqslant 1.
    \end{equation}
\end{theorem}

\begin{proposition}[Uniform Kurdyka-\L{}ojasiewicz inequality {\cite[Proposition 5]{josz2023global}}] 
\label{prop:UKL}
Let $f:\mathbb{R}^n\rightarrow \mathbb{R}$ be locally Lipschitz and semialgebraic. Let $X$ be a bounded subset of $\mathbb{R}^n$ and $V$ be the set of critical values of $f$ in $\overline{X}$ if it is non-empty, otherwise $V:=\{0\}$. There exists a concave semialgebraic diffeomorphism $\psi:[0,\infty)\rightarrow[0,\infty)$ such that
\begin{equation}
\label{eq:uniform_KL}
    \forall x \in X \setminus (\partial \tilde{f})^{-1}(0), ~~~ d(0,\partial (\psi \circ \tilde{f})(x)) \geqslant 1,
\end{equation}
where $\tilde{f}(x) := d(f(x),V)$ for all $x\in\mathbb{R}^n$.
\end{proposition}

We say that $\psi:[0,\infty)\rightarrow[0,\infty)$ in Proposition  \ref{prop:UKL} is a desingularizing function of $f$ on $X$. The uniform Kurdyka-\L{}ojasiewicz inequality \eqref{eq:uniform_KL} enables one to relate the length of the iterates of the gradient method in any bounded region to the function variation \cite[Proposition 8]{josz2023global}. If one uses the Kurdyka-\L{}ojasiewicz inequality \eqref{eq:KL} instead, then the function values evaluated at the iterates would be restricted to a potentially small range around a critical value. If one uses the uniformized KL property \cite[Lemma 6]{bolte2014proximal}, then the iterates would need to lie in a uniform neighborhood of a compact subset of the critical points of $f$ where $f$ is constant.  

In order to prove \cite[Proposition 8]{josz2023global}, one uses the fact that the objective function is a Lyapunov function for all sufficiently small step sizes in the gradient method. However, in the momentum method the Lyapunov function depends on the step size, as can be seen in Lemma \ref{lemma:lyapunov}. The main challenge that we thus face is to obtain an upper bound on the length of the iterates that is independent of the step size. Otherwise, it could blow up as the step size gets small. Such is the object of the following results.

Proposition \ref{prop:uniform_decrease} takes a first step by showing that the length is bounded by a constant times a desingularizing function evaluated at the Lyapunov function variation plus a constant multiple of the step size. Both constants are independent of the step size. Proposition \ref{prop:lyapunov_uniformKL} ensures that the desingularizing function no longer depends on the step size. Finally, Lemma \ref{lemma:length_formula} gets rid of the dependence on the step size in the argument of the desingularizing function.

Proposition \ref{prop:uniform_decrease} below generalizes \cite[Proposition 8]{josz2023global} from the gradient method to the momentum method. While in the gradient method we have $c_3 = 2$ in \eqref{eq:length_discrete}, in the momentum method obtaining an expression for $c_3$ that does not depend on the step size requires some care. Let $\mathbb{N}^* := \{1,2,3,\hdots\}$. We will use the following simple lemma.

\begin{lemma}
    \label{lemma:concave}
    If $h:\mathbb{R}_+\rightarrow\mathbb{R}_+$ is concave, $h(0) = 0$, and $a,b\geqslant 0$, then $|h(b)-h(a)| \leqslant h(|b-a|)$. 
\end{lemma}
\begin{proof}
    We first show that $h$ is increasing. Since $h$ is concave, for any $0 \leqslant x < y < z$, we have
    \begin{equation*}
        h(z) \leqslant \frac{h(y)-h(x)}{y-x}(z-y)+h(y). 
    \end{equation*}
    If $(h(y)-h(x))(y-x)<0$, then we obtain the contradiction $0 \leqslant \limsup_{z\to\infty}h(z)=-\infty$, establishing that $h$ is increasing. It thus suffices to show that $h(b)-h(a)\leqslant h(b-a)$ for any $0 \leqslant a \leqslant b$. Since $h$ is concave, we have
    \begin{equation*}
        \frac{a}{b}h(0) + \frac{b-a}{b}h(b) \leqslant h(b-a) ~~~\text{and}~~~  \frac{b-a}{b}h(0) + \frac{a}{b}h(b) \leqslant h(a).
    \end{equation*}
    Summing these two inequalities and using the fact that $h(0)=0$ yields the desired result. 
\end{proof}

\begin{proposition}
\label{prop:uniform_decrease}
Let $f\in C^{1,1}_{\rm loc}(\mathbb{R}^n)$ be semialgebraic, $X\subset\mathbb{R}^n$ be bounded, $\beta\in (-1,1)$, $\gamma\in \mathbb{R}$, $\delta \geqslant 0$, and $m \in \mathbb{N}^*$ be an upper bound on the number of critical values of $f$ in $\overline{X}$. There exist $\bar{\alpha},c_3,\zeta>0$ such that for all $\alpha\in(0,\bar{\alpha}]$, there exists $\lambda>0$ such that, for any desingularizing function $\psi_\lambda$ of $H_\lambda$ on $X \times X$ and for all $K\in\mathbb{N}$, if $(x_{k})_{k\in \{-1\}\cup\mathbb{N}}$ are iterates of the momentum method \eqref{eq:momentum} for which $x_{-1},\hdots,x_K\in X$ and $\|x_0-x_{-1}\|\leqslant \delta\alpha$, then
\begin{equation}
\label{eq:length_discrete}
   \frac{1}{2m}\sum_{k=0}^{K}\|z_{k+1}-z_k\| ~\leqslant~ c_3\psi_\lambda\left(\frac{H_\lambda(z_0)-H_\lambda(z_K)}{2m}\right)+\zeta\alpha
\end{equation}
and $H_\lambda(z_0)\geqslant \cdots \geqslant H_\lambda(z_K)$ where $z_k:=(x_k,x_{k-1})\in\mathbb{R}^{2n}$. If $L>0$ and $M>0$ are Lipschitz constants of $\bar{\beta}f$ and $\nabla f$ respectively on $S+\max\{|\beta|,|\gamma|\}(S-S)$ where $S$ is the convex hull of $X$ and $\bar{\beta}:= (1-\beta)^{-1}$, then one may take the same $\bar{\alpha}$ as in \eqref{eq:alpha_range}, 
\begin{equation}
\label{eq:c_3_lambda}
    c_3 := \frac{8\sqrt{2}(2+|\gamma|+3|\beta|)}{1-\beta^2},\; \zeta:=2\sqrt{2}(\delta+L),\text{ and }\lambda := \frac{\beta^2+1+M\beta^2\alpha}{4\alpha}.
\end{equation}
\end{proposition}
\begin{proof}
Consider $\bar{\alpha}$ and $c_3$ as defined in \eqref{eq:alpha_range} and \eqref{eq:c_3_lambda} respectively. Given $\alpha\in(0,\bar{\alpha}]$, let $\lambda \in (\lambda^-,\lambda^+)$ where $\lambda^-$ and $\lambda^+$ are defined in \eqref{eq:lambda_range}. Let $\psi_\lambda$ be a desingularizing function of $H_\lambda$ on $X\times X$. By Lemma \ref{lemma:lyapunov} and \ref{lem:gradupper}, for $k=0,\ldots,K-1$ we have 
\begin{equation}
    H_\lambda(z_{k+1}) - H_\lambda(z_k) \leqslant - c_1\|z_{k+1} - z_k\|^2 \leqslant -\frac{c_1}{c_2} \|\nabla H_\lambda(z_k)\|\|z_{k+1} - z_k\| \label{eq:descent_lemma}
\end{equation}
and
\begin{equation}
    H_\lambda(z_{k+1}) - H_\lambda(z_k) \leqslant - c_1\|z_{k+1} - z_k\|^2 \leqslant -\frac{c_1}{c_2} \|\nabla H_\lambda(z_{k+1})\|\|z_{k+1} - z_k\|. \label{eq:descent_lemma_k+1}
\end{equation}

Since $\nabla H_\lambda(x,y) = (\nabla f(x)+2\lambda(x-y) , 2\lambda(y-x))^\top$, the critical values of $f$ in $\overline{X}$ are the same as those of $H_\lambda$ in $\overline{X\times X}$. We let $V$ denote this set of critical values if they exist, otherwise $V:=\{0\}$.

Assume that $[H_\lambda(z_K),H_\lambda(z_0))$ excludes the elements of $V$ and the averages of any two consecutive elements of $V$.\footnote{The point of excluding elements in $V$ and the averages of two consecutive elements in $V$ is to guarantee that there is a unique closest element in $V$ that works for all $H_\lambda(z_K),\ldots,H_\lambda(z_0)$ and this element is either greater than or equal to all of them or less than all of them.} If $H_\lambda(z_1)=H_\lambda(z_0)$, then $z_1 = z_0$ by \eqref{eq:descent_lemma}. Thus $\nabla f(x_0) = 0$ and $z_K = \dots = z_0$ by induction. Otherwise, we have that $H_\lambda(z_1)<H_\lambda(z_0)$. With $\widetilde{H}_\lambda := d(H_\lambda,V)$, we thus have $0 \not \in \partial \widetilde{H}_\lambda(z_k)$ and $1 \leqslant \| \nabla(\psi_\lambda \circ \widetilde{H}_\lambda)(z_k)\| = \psi_\lambda'(\widetilde{H}_\lambda(z_k)) \|\nabla \widetilde{H}_\lambda(z_k)\|$ for $k=1,\hdots,K$ by the uniform Kurdyka-\L{}ojasiewicz inequality \eqref{eq:uniform_KL}. Let $k \in \{1,\ldots K-1\}$. If $\widetilde{H}_\lambda (z_k) \geqslant \widetilde{H}_\lambda(z_{k+1})$, multiplying \eqref{eq:descent_lemma} by $\psi_\lambda'(\widetilde{H}_\lambda(z_k))$ and using concavity of $\psi_\lambda$, we find that  
\begin{align*}
     \|z_{k+1}-z_k\| & \leqslant \frac{c_2}{c_1}\psi_\lambda'(\widetilde{H}_\lambda(z_k))(H_\lambda(z_k) - H_\lambda(z_{k+1})) \\
    & = \frac{c_2}{c_1}\psi_\lambda'(\widetilde{H}_\lambda(z_k))(\widetilde{H}_\lambda(z_k) - \widetilde{H}_\lambda(z_{k+1})) \\
    & \leqslant \frac{c_2}{c_1}( \psi_\lambda (\widetilde{H}_\lambda(z_k)) - \psi_\lambda (\widetilde{H}_\lambda(z_{k+1}))).
\end{align*}
If $\widetilde{H}_\lambda (z_k) \leqslant \widetilde{H}_\lambda(z_{k+1})$, multiplying \eqref{eq:descent_lemma_k+1} by $\psi_\lambda'(\widetilde{H}_\lambda(z_{k+1}))$ and using concavity of $\psi_\lambda$, we find that
\begin{align*}
    \|z_{k+1}-z_k\| &\leqslant \frac{c_2}{c_1}\psi_\lambda'(\widetilde{H}_\lambda(z_{k+1}))(H_\lambda(z_k) - H_\lambda(z_{k+1})) \\
    & = \frac{c_2}{c_1}\psi_\lambda'(\widetilde{H}_\lambda(z_{k+1}))(\widetilde{H}_\lambda(z_{k+1}) - \widetilde{H}_\lambda(z_k)) \\
    & \leqslant \frac{c_2}{c_1}( \psi_\lambda (\widetilde{H}_\lambda(z_{k+1})) - \psi_\lambda (\widetilde{H}_\lambda(z_{k}))).
\end{align*}
As a result,
\begin{equation*}
    \|z_{k+1}-z_k\| \leqslant \frac{c_2}{c_1} |\psi_\lambda (\widetilde{H}_\lambda(z_k)) - \psi_\lambda (\widetilde{H}_\lambda(z_{k+1}))|, ~~~ k = 1,\hdots,K-1.
\end{equation*}
We obtain the telescoping sum 
\begin{subequations}
\label{eq:length_finite}
\begin{align}
    \sum_{k=0}^K \|z_{k+1}-z_k\| &\leqslant \|z_1 - z_0\| +\sum_{k=1}^{K-1} \frac{c_2}{c_1}\left| \psi_\lambda \left(\widetilde{H}_\lambda(z_k)\right) - \psi_\lambda \left(\widetilde{H}_\lambda(z_{k+1})\right) \right| + \|z_{K+1} - z_{K}\| \label{eq:length_finite_a}\\
    & = \|z_1 - z_0\| +\frac{c_2}{c_1} \left|\psi_\lambda\left(\widetilde{H}_\lambda(z_0)\right) - \psi_\lambda\left(\widetilde{H}_\lambda(z_K)\right)\right| + \|z_{K+1}-z_K\| \label{eq:length_finite_b}\\
    & \leqslant \sqrt{2}(\delta+L)\alpha + \frac{c_2}{c_1} \left(\psi_\lambda \left( \left| \widetilde{H}_\lambda(z_0) - \widetilde{H}_\lambda(z_K) \right|\right) - \psi_\lambda(0)\right)+ \sqrt{2}(\delta+L)\alpha \label{eq:length_finite_c} \\
    & = \frac{c_2}{c_1}\psi_\lambda( H_\lambda(z_0) - H_\lambda(z_K))+\zeta\alpha\label{eq:length_finite_d}
\end{align}
\end{subequations}
where $\zeta$ is defined in \eqref{eq:c_3_lambda}. Above, \eqref{eq:length_finite_b} and \eqref{eq:length_finite_d} are due to the monotonicity of $\widetilde{H}_\lambda(z_0), \ldots , \widetilde{H}_\lambda(z_K)$. We use Lemma \ref{lemma:concave} and Fact \ref{fact:O(alpha)} to obtain \eqref{eq:length_finite_c}.

We next consider the general case where
\begin{equation*}
    [H_\lambda(z_K),H_\lambda(z_{K_p+1})) \cup \hdots \cup [H_\lambda(z_{K_2}),H_\lambda(z_{K_1+1})) \cup [H_\lambda(z_{K_1}),H_\lambda(z_0))
\end{equation*}
excludes the elements of $V$ and the averages of any two consecutive elements of $V$. For notational convenience, let $K_0 := -1$ and $K_{p+1} := K$. Since $p \leqslant 2m-1$, we have
\begin{subequations}
\begin{align}
    \sum_{k=0}^K \|z_{k+1}-z_k\| & = \sum_{i=0}^p \sum_{k=K_i+1}^{K_{i+1}} \|z_{k+1}-z_k\| \\
    & \leqslant \sum_{j=0}^p \left(\frac{c_2}{c_1} \psi_\lambda( H_\lambda(z_{{K_i}+1}) - H_\lambda(z_{K_{i+1}}) + \zeta\alpha \right) \\
    & \leqslant \frac{c_2}{c_1} \sum_{i=0}^p \psi_\lambda(H_\lambda(z_{{K_i}+1}) - H_\lambda(z_{K_{i+1}})) + (p+1)\zeta\alpha \\
    & \leqslant \frac{c_2}{c_1}(p+1) ~ \psi_\lambda \left( \frac{1}{p+1} \sum_{i=0}^p (H_\lambda(z_{{K_i}+1}) - H_\lambda(z_{K_{i+1}})) \right) \\ 
    & + (p+1)\zeta\alpha \label{eq:unidecr_jensen} \\
    & \leqslant \frac{c_2}{c_1}(p+1) ~ \psi_\lambda \left( \frac{H_\lambda(z_0)-H_\lambda(z_K)}{p+1} \right) + (p+1)\zeta\alpha \\
    & \leqslant \frac{c_2}{c_1}2m ~ \psi_\lambda \left( \frac{L(z_0)-L(z_K)}{2m} \right) + 2m\zeta\alpha. \label{eq:unidecr_mono}
\end{align}
\end{subequations}
Indeed, \eqref{eq:unidecr_jensen} follows from Jensen's inequality and \eqref{eq:unidecr_mono} follows from the fact that $s\mapsto s\psi_\lambda(a/s)$ is increasing over $(0,\infty)$ for any constant $a>0$. 
Substituting $c_1$ and $c_2$ using \eqref{eq:c_1}, \eqref{eq:lambda_range}, and \eqref{eq:c_2}, we find that
\begin{align*}
    \frac{c_2}{c_1} &= \frac{\sqrt{2}\max\left\{\frac{1}{\alpha},\frac{|\beta|}{\alpha}+M(|\gamma|+1)+4\lambda\right\}}{\min\left\{\lambda-\left(\frac{1}{2\alpha}+\frac{M}{2}\right)\beta^2-\frac{|\beta-\gamma|M}{2}, \frac{1}{2\alpha}-\frac{|\beta-\gamma|M}{2}-\lambda\right\}} \\
    &= \frac{2\sqrt{2}\max\{1,|\beta|+M(|\gamma|+1)\alpha+4\lambda\alpha\}}{\min\left\{2\lambda\alpha-\left(1+\alpha M\right)\beta^2-|\beta-\gamma|M\alpha,1-|\beta-\gamma|M\alpha-2\lambda\alpha\right\}}. 
\end{align*}
If we take $\lambda$ to be the midpoint of $(\lambda^-,\lambda^+)$, i.e., $\lambda=(\beta^2+1+M\beta^2\alpha)/(4\alpha)$, then this simplifies to
\begin{equation*}
    \frac{c_2}{c_1} = \frac{4\sqrt{2}(|\beta|+M(|\gamma|+1)\alpha+\beta^2+1+M\beta^2\alpha)}{1-\beta^2-(\beta^2+2|\beta-\gamma|)M\alpha}.
\end{equation*}
Notice that $c_2/c_1$ is a increasing function of $\alpha$ over $(0,\bar{\alpha}]$, where we recall that $\bar{\alpha}=\min\{1/M,(1-\beta^2)/(2(\beta^2+2|\beta-\gamma|)M)\}$. As a result, 

\begin{align*}
    \frac{c_2}{c_1} &\leqslant \frac{4\sqrt{2}(|\beta|+M(|\gamma|+1)\bar{\alpha}+\beta^2+1+M\beta^2\bar{\alpha})}{1-\beta^2-(\beta^2+2|\beta-\gamma|)M\bar{\alpha}} \\ 
    &\leqslant \frac{4\sqrt{2}(|\beta|+M(|\gamma|+1)\frac{1}{M}+\beta^2+1+M\beta^2\frac{1}{M})}{1-\beta^2-(\beta^2+2|\beta-\gamma|)M\frac{1-\beta^2}{2(\beta^2+2|\beta-\gamma|)M}} \\ 
    &= \frac{4\sqrt{2}(|\beta|+|\gamma|+1+\beta^2+1+\beta^2)}{(1-\beta^2)/2} \\ 
    &\leqslant \frac{8\sqrt{2}(2+|\gamma|+3|\beta|)}{1-\beta^2}=:c_3 > 0.\qedhere
\end{align*}

\end{proof}

If the objective function satisfies the \L{}ojasiewicz gradient inequality, then the Lyapunov function also satisfies it according to \cite[Theorem 3.6]{li2018calculus}. Proposition \ref{prop:lyapunov_uniformKL} below generalizes \cite[Theorem 3.6]{li2018calculus} from functions satisfying the \L{}ojasiewicz gradient inequality to functions satisfying the uniform Kurdyka-\L{}ojasiewicz inequality. We show that a suitable choice of desingularizing function for the objective is a common desingularizing function for the Lyapunov functions for all sufficiently large parameters. 

\begin{proposition}
    \label{prop:lyapunov_uniformKL}
    Let $f:\mathbb{R}^n\rightarrow \mathbb{R}$ be a locally Lipschitz semialgebraic function and $X\subset\mathbb{R}^n$ be bounded. The family of functions $(H_\lambda)_{\lambda \geqslant 1/4}$ admits a common desingularizing function on $X\times X$.
\end{proposition}

\begin{proof} 
By Proposition \ref{prop:UKL}, there exists a desingularizing function $\psi$ of $f$ on $X$. Without loss of generality, we may assume that $\psi'(t) \geqslant 1/\sqrt{t}$ for all $t>0$, after possibly replacing $\psi$ by $t\mapsto \int_0^t \max\{\psi'(s),1/\sqrt{s}\}ds$, which is semialgebraic\footnote{To see why, note that $\{s >0 : \psi'(s)\geqslant 1/\sqrt{s}\}$ is semialgebraic and hence a finite union of open intervals and points. Thus the integral is equal to $\psi$ up to a constant on finitely many intervals of $\mathbb{R}_+$, and $t\mapsto 2\sqrt{t}$ up to a constant otherwise. The graph of such a function is hence semialgebraic.} and concave since the integrand is decreasing. We may also multiply $\psi$ by $1/\min\{1,c\}\geqslant 1$ where
\begin{equation}
\label{eq:v1+v2}
    c:= \inf\left\{ \psi'((v_2-v_1)/2)\theta((v_1+v_2)/2): v_1,v_2\in V, ~ v_1<v_2, ~ (v_1,v_2)\cap V = \emptyset\right\} >0,
\end{equation}
$\theta(v):= \inf \{ d(0,\partial f(x)) : x\in X, f(x) = v \}$ for all $v \in \mathbb{R}$, and $V$ is the set of critical values of $f$ in $\overline{X}$ if it is non-empty, otherwise $V:=\{0\}$. Note that $c>0$ as it is the infimum of finitely many positive real numbers. Indeed, $\psi'(t)>0$ for all $t>0$ and $\theta(v)>0$ for all $v \not\in V$. To see why the latter statement holds, assume the contrary that $\theta(v) = 0$ for some $v \not\in V$. Then there exists $(x_k,s_k)_{k\in \mathbb{N}} \subset X \times \mathbb{R}^n$ such that $f(x_k) = v$, $s_k \in \partial f(x_k)$, and $s_k \rightarrow 0$. As $X$ is bounded, $(x_k)_{k\in \mathbb{N}}$ admits a limit point $\bar{x}$. We have that $f(\bar{x}) = v$ by continuity of $f$ and $0 \in \partial f(\bar{x})$ by \cite[2.1.5 Proposition (b)]{clarke1990}. Thus $v \in V$ and a contradiction occurs.

By \cite[Corollary 1, p. 39]{clarke1990} we have 
\begin{equation}
\label{eq:partial_L_lambda}
    \partial H_\lambda(x,y)  = \left(\partial f(x) + \{2\lambda(x-y)\}\right)\times \{2\lambda (y - x)\},    
\end{equation}
so that $0\in \partial H_\lambda(x,y)$ if and only if $0 \in \partial f(x)$ and $x=y$. Therefore, the set of critical values of $f$ in $\overline{X}$ and the set of critical values of $H_\lambda$ in $\overline{X\times X}$ coincide. Accordingly, let $\tilde{f}:=d(f,V)$ and $\widetilde{H}_\lambda:=d(H_\lambda,V)$. Now fix $\lambda \geqslant 1/4$. For all $x,y\in X$ such that $0 \notin \partial \widetilde{H}_\lambda(x,y)$, we have
\begin{subequations}
    \label{eq:L_lambda}
    \begin{align}
        d(0,\partial \widetilde{H}_\lambda (x,y)) & = d(0,\partial H_\lambda(x,y)) \label{eq:L_lambda_a}\\
        & = d\left(0,\left(\partial f(x) + \{2\lambda(x-y)\}\right)\times \{2\lambda (y - x)\}\right) \label{eq:L_lambda_b}\\
        & = \sqrt{d(0, \partial f(x) + 2\lambda(x-y))^{2} + \|2\lambda(x-y)\|^2} \label{eq:L_lambda_c}\\
        & \geqslant \sqrt{\eta_1 d(0,\partial f(x))^2 - \eta_2\|2\lambda(x-y)\|^{2} + \|2\lambda(x-y)\|^2}\label{eq:L_lambda_d}\\
        & = \sqrt{ \eta_1 d(0,\partial f(x))^2 + (1-\eta_2)4\lambda^2\|x-y\|^2} \label{eq:L_lambda_e}\\
        & \geqslant \sqrt{\eta_1 d(0,\partial f(x))^2 + (1-\eta_2)\lambda\|x-y\|^2} \label{eq:L_lambda_f}\\
        & \geqslant \sqrt{\frac{\eta_1}{\psi'(\tilde{f}(x))^2} + \frac{1-\eta_2}{\psi'(\lambda\|x-y\|^2)^2}} \label{eq:L_lambda_g}\\
        & \geqslant \sqrt{\frac{\min\{\eta_1,1-\eta_2\}}{\psi'(\max\{\tilde{f}(x),\lambda\|x-y\|^2\})^2}} \label{eq:L_lambda_h}\\
        & \geqslant \sqrt{ \frac{\min\{\eta_1,1-\eta_2\}}{\psi'\left(\frac{\tilde{f}(x)+\lambda\|x-y\|^2}{2}\right)^2}} \label{eq:L_lambda_i}\\
        & \geqslant \frac{\sqrt{\min\{\eta_1,1-\eta_2\}}}{\psi'\left(\frac{\widetilde{H}_\lambda(x,y)}{2}\right)}. \label{eq:L_lambda_j}
    \end{align}
\end{subequations}
Above, \eqref{eq:L_lambda_a} holds because $0 \notin \partial \widetilde{H}_\lambda(x,y)$ and thus $\widetilde{H}_\lambda(x',y') - \widetilde{H}_\lambda(x,y) = \pm (H_\lambda(x',y') - H_\lambda(x,y))$ for all $(x',y')$ in neighborhood of $(x,y)$ where the sign is constant. \eqref{eq:L_lambda_b} is due to \eqref{eq:partial_L_lambda}. \eqref{eq:L_lambda_c} holds because the distance function is defined using the Euclidean norm. The existence of the constants $\eta_1>0$, $\eta_2\in (0,1)$ in \eqref{eq:L_lambda_d} are guaranteed by \cite[Lemma 3.1]{li2018calculus}. \eqref{eq:L_lambda_e} comes from a factorization. \eqref{eq:L_lambda_f} is due to the fact that $\lambda \geqslant 1/4$. \eqref{eq:L_lambda_g} is due to the uniform Kurdyka-\L{}ojasiewicz inequality \eqref{eq:uniform_KL} and the fact that $\psi'(t) \geqslant 1/\sqrt{t}$ for all $t>0$. Indeed, if $0 \notin \partial \tilde{f}(x)$, then $d(0,\partial f(x)) = d(0,\partial \tilde{f}(x)) \geqslant 1/\psi'(\tilde{f}(x))$ by \cite[2.3.9 Theorem (Chain Rule I) (ii) p. 42]{clarke1990}. If $0 \in \partial \tilde{f}(x)$, then $\tilde{f}(x) = 0$ or $f(x) = (v_1+v_{2})/2$ for some $v_1,v_2 \in V$ such that $v_1< v_2$ and $(v_1,v_2)\cap V = \emptyset$. In the former case, $d(0,\partial f(x)) \geqslant 1/\psi'(\tilde{f}(x)) = 1/\psi'(0) = 1/\infty = 0$ where $\psi'(0):=\lim_{a \searrow 0}\psi'(a)$. In the latter case, we have $\tilde{f}(x) = (v_2-v_1)/2$ and thus $d(0,\partial f(x)) \geqslant \theta((v_1+v_2)/2)\geqslant 1/\psi'((v_2-v_1)/2) = 1/\psi'(\tilde{f}(x))$ by \eqref{eq:v1+v2}. \eqref{eq:L_lambda_h} and \eqref{eq:L_lambda_i} hold because $\psi$ is concave and thus $\psi'$ is decreasing. \eqref{eq:L_lambda_j} is due to the fact that $0<\widetilde{H}_\lambda(z) = d(f(x)+\lambda\|x-y\|^2,V) \leqslant d(f(x),V)+\lambda\|x-y\|^2 = \tilde{f}(x)+\lambda\|x-y\|^2$. We conclude that $t \in [0,\infty) \rightarrow 2\psi(t/2)/\sqrt{\min\{\eta_1,1-\eta_2\}}$ is a desingularizing function of $H_\lambda$ on $X \times X$ for all $\lambda \geqslant 1/4$, which is actually also a desingularizing function of $f$ on $X$.
\end{proof}

Thanks to Proposition \ref{prop:uniform_decrease} and Proposition \ref{prop:lyapunov_uniformKL}, we are now ready to prove the length formula.

\begin{proof}[Proof of Lemma \ref{lemma:length_formula}]
Let $m\in \mathbb{N}^*$ be an upper bound of the number of critical values of $f$ in $\overline{X}$. We apply Proposition \ref{prop:uniform_decrease} to the set $X$ and let $\bar{\alpha} \in (0,1]$, $c_3>1$, and $\zeta>0$ be given by the proposition. Let $\bar{\psi}$ be a common desingularizing function of $(H_\lambda)_{\lambda\geqslant 1/4}$ on $X \times X$ given by Proposition \ref{prop:lyapunov_uniformKL}. Let $\alpha \in (0,\bar{\alpha}]$ and let $\lambda := (\beta^2+1+M\beta^2\alpha)/(4\alpha)\geqslant 1/4$ as defined in \eqref{eq:c_3_lambda}.  Since $c_3>1$, $\psi(t):=2c_3m\bar{\psi}(t/(2m))$ is also a desingularizing function of $H_\lambda$ on $X\times X$. Let $\kappa := 2m \zeta$ and $\eta: = 2m\delta^2(\beta^2+1+M\beta^2)/4$ where $M>0$ is a Lipschitz constant of $\nabla f$ on $S+\max\{|\beta|,|\gamma|\}(S-S)$ and $S$ is the convex hull of $X$. It follows from \eqref{eq:length_discrete} that 
\begin{align*}
   \sum_{k=0}^{K}\|x_{k+1}-x_k\| & \leqslant  \psi\left(H_\lambda(x_0,x_{-1})-H_\lambda(x_{K},x_{K-1})\right) + \kappa\alpha  \\[-2mm]
   & \leqslant \psi\left(f(x_0)-f(x_{K}) + \lambda\|x_0-x_{-1}\|^2\right) + \kappa\alpha \\[1mm]
   & \leqslant \psi\left(f(x_0)-f(x_{K}) + \lambda \delta^2\alpha^2\right) + \kappa\alpha \\[1mm]
   & \leqslant \psi\left(f(x_0)-f(x_{K}) + \eta\alpha\right)+ \kappa\alpha. \qedhere
\end{align*}
\end{proof}

\section{Conclusion}
\label{sec:Conclusion}

This work departs from the commonly accepted assumptions in the literature, and hence guarantees convergence of the momentum method for a larger class of functions which includes important problems in data science. Future directions include extending the analysis to nonsmooth and stochastic settings. 

\section*{Acknowledgments}
We thank the reviewers and the associate editor for their valuable feedback.

\appendix

\section{Proof of Fact \ref{fact:gradupper}}
\label{sec:Proof of Fact fact:gradupper}
By definition of $y_k^\beta$ and $y_k^\gamma$ in (\ref{eq:sepalg}), for $k=0,\hdots,K$, we have
\begin{align*}
    \|\nabla f(x_k)\| &\leqslant \|\nabla f(y_k^\gamma)\| + \|\nabla f(x_k) - \nabla f(y_k^\gamma)\| \\
    &\leqslant \|x_{k+1}-y_k^\beta\|/\alpha + M|\gamma| \|x_k - x_{k-1}\| \\
    &\leqslant \|x_{k+1} - x_k\|/\alpha + (|\beta|/\alpha +M|\gamma|) \|x_k - x_{k-1}\| \\
    &\leqslant \sqrt{2}\max\{1/\alpha,|\beta|/\alpha+M|\gamma|\}\|z_{k+1} - z_k\|,
\end{align*}

\section{Proof of Lemma \ref{lemma:tracking}}
\label{sec:Proof of Lemma lemma:tracking} The matrix
\begin{equation*}
    A: = \begin{pmatrix}
        1+\beta & -\beta \\ 
        1 & \hphantom{-}0
\end{pmatrix} \otimes I_n
\end{equation*}
is diagonalizable as the Kronecker product of two such matrices \cite[Exercise 15 p. 265]{horn2012matrix}. Thus there exist an invertible matrix $P \in \mathbb{R}^{2n\times 2n}$ and a diagonal matrix $D\in \mathbb{R}^{2n\times 2n}$ such that $A = PDP^{-1}$. Let $\|x\|_P := \|P^{-1}x\|$. Then for any $X\in \mathbb{R}^{2n\times 2n}$, we have $\|X\|_P : = \sup \{ \|Xv\|_P : \|v\|_P \leqslant 1\}  = \sup \{ \|P^{-1}XPv\| : \|v\| \leqslant 1\} = \|P^{-1}XP\|$.
By equivalence of norms, there exist $c_1,c_2,c_3>0$ such that $c_1\|x\| \leqslant \|x\|_P \leqslant c_2\|x\|$ and $\|X\|_P \leqslant c_3\|X\|$. 
    
Similar to \cite[Theorem 2]{kovachki2021continuous}, let $\bar{\beta}:=1/(1-\beta)$. Also, let $L>0$ and $M>0$ respectively denote Lipschitz constants of $\bar{\beta}f$ and $\bar{\beta}\nabla f$ with respect to the Euclidean norm $\|\cdot\|$ on $S+\gamma(S-S)$ where $S$ denotes the convex hull of  $B(X_0,\sigma_T(X_0)+\delta+1):= X_0 + B(0,\sigma_T(X_0)+\delta+1)$ and
    \begin{subequations}
\label{eq:sup_trajectory_T}
\begin{align}
    \sigma_T(X_0) := & \sup\limits_{x \in C^1(\mathbb{R}_+,\mathbb{R}^n)} ~~ \int_0^T \|x'(t)\|dt \\
  & ~~ \mathrm{subject~to} ~~~ 
\left\{ 
\begin{array}{l}
x'(t) = -\bar{\beta}\nabla f(x(t)), ~\forall t > 0,\\[3mm] x(0) \in X_0.
\end{array}
\right.
\end{align}
\end{subequations}
    Let $c_4 := ML(1/2+|\beta|/2+|\gamma|-\beta|\gamma|)$ and $c_5 := c_2 M\sqrt{1+ 2\gamma + 2\gamma^2}$. Without loss of generality, we may assume that $X_0 \neq \emptyset$. The feasible set of \eqref{eq:sup_trajectory_T} is thus non-empty (i.e., $\sigma_T(X_0)>-\infty$) because $f$ is lower bounded and belongs to $C^{1,1}_{\mathrm{loc}}(\mathbb{R}^n)$ \cite[Theorem 17.1.1]{attouch2014variational}. Notice that we also have $\sigma_T(X_0)<\infty$. Indeed, by the Cauchy-Schwarz inequality any feasible point $x(\cdot)$ of \eqref{eq:sup_trajectory_T} satisfies

\begin{align*}
    \int_0^T \|x'(t)\|dt & \leqslant \sqrt{T} \sqrt{ \int_0^T \|x'(t)\|^2 dt} \\
    & = \sqrt{T} \sqrt{ \int_0^T \langle - \bar{\beta}\nabla f(x(t)),x'(t)\rangle dt} \\
    & = \sqrt{T} \sqrt{ \bar{\beta}f(x(0)) - \bar{\beta}f(x(T))} \\
    & \leqslant \sqrt{T\bar{\beta}\left(\sup_{X_0} f - \inf_{\mathbb{R}^n} f\right)}<\infty. \label{eq:right_below}
\end{align*}

It is easy to check that $L$ and $ML$ are respectively Lipschitz and gradient Lipschitz constants on $[0,T]$ of any feasible point $x(\cdot)$ of \eqref{eq:sup_trajectory_T}. Indeed, let $x(\cdot)$ be a feasible point of \eqref{eq:sup_trajectory_T}. Since $x(t) \in B(X_0,\sigma_T(X_0))$ for all $t\in [0,T]$, we have $\|x'(t)\| = \|\bar{\beta}\nabla f(x(t))\| \leqslant L$. By the mean value theorem, for all $s,t \in [0,T]$ we have $\|x'(t)-x'(s)\| = \|\bar{\beta}\nabla f(x(t)) - \bar{\beta}\nabla f(x(s))\| \leqslant M \|x(t)-x(s)\| \leqslant M L | t-s |$. As a byproduct, we get the Taylor bound
\begin{equation}
    \label{eq:local_error}
    \|x(t)-x(s) - x'(s)(t-s)\| \leqslant \frac{ML}{2}(t-s)^2.
\end{equation}

Let $\epsilon \in (0,\delta+1]$ and $\bar{\alpha} := \min\{1,\epsilon c_1c_2^{-1}[e^{c_5 T} (|\beta|\delta + 2L-L\beta +c_4c_5^{-1}) -c_4c_5^{-1})]^{-1}\}>0$. Let $x_{-1},x_0, x_1, \ldots \in \mathbb{R}^n$ be a sequence generated by the gradient method with momentum and step size $\alpha \in (0, \bar{\alpha}]$ for which $x_0\in X_0$ and $\|x_0-x_{-1}\|\leqslant \delta\alpha$. Let $x(\cdot)$ be a feasible point of \eqref{eq:sup_trajectory_T} such that $x(0) = x_0$. Similar to the proof of \cite[Theorem 2]{kovachki2021continuous}, let $\bar{x}_k := x(k\alpha)$ for all $k \in \mathbb{N}$. We next reason by induction. We have $\|x_0-\bar{x}_0\| = 0 \leqslant \epsilon$. Assume that $\|x_k-\bar{x}_k\| \leqslant \epsilon$ for $k=0,\hdots,K$ for some index $K\in \mathbb{N}$. For $k=1,\hdots,K$, we have
\begin{subequations}
\begin{align}
    \|\bar{x}_{k+1} - \bar{x}_k + \alpha \bar{\beta}\nabla f(\bar{x}_k) \| & \leqslant  ML\alpha^2/2, \label{eq:taylor1} \\
    \|\bar{x}_{k-1} - \bar{x}_k - \alpha \bar{\beta}\nabla f(\bar{x}_k) \| & \leqslant  ML\alpha^2/2. \label{eq:taylor2}
\end{align}
\end{subequations}
Multiplying \eqref{eq:taylor2} by $|\beta|$ and adding it to \eqref{eq:taylor1} yields 
\begin{equation*}
    \|\bar{x}_{k+1} - \bar{x}_{k} - \beta(\bar{x}_{k}-\bar{x}_{k-1}) + \alpha \nabla f(\bar{x}_{k}) \| \leqslant ML(1+|\beta|)\alpha^2/2,    
\end{equation*}
where we use the fact that $\bar{\beta} - \beta\bar{\beta} = 1$. We also have
\begin{align*}
    \|\nabla f(\bar{x}_{k}+\gamma(\bar{x}_{k}-\bar{x}_{k-1})) - \nabla f(\bar{x}_{k})\| & \leqslant M(1-\beta)|\gamma| \|\bar{x}_{k}-\bar{x}_{k-1}\| \\ & \leqslant ML(1-\beta)|\gamma| \alpha. 
\end{align*}
Hence by combining the above two inequalities, we have 
\begin{equation*}
    \| \bar{x}_{k+1} -\bar{x}_{k} - \beta(\bar{x}_{k}-\bar{x}_{k-1}) + \alpha \nabla f(\bar{x}_{k}+\gamma(\bar{x}_{k}-\bar{x}_{k-1}))\| \leqslant c_4\alpha^2
\end{equation*}
where $c_4 = ML(1/2+|\beta|/2+|\gamma|-\beta|\gamma|)$. Let $e_k = x_k - \bar{x}_k$. We have
\begin{equation*}
    \|e_{k+1} -e_{k} - \beta(e_{k}-e_{k-1}) + \alpha [\nabla f(x_{k}+\gamma(x_{k}-x_{k-1})) - \nabla f(\bar{x}_{k}+\gamma(\bar{x}_{k}-\bar{x}_{k-1})) ]\| \leqslant c_4\alpha^2
\end{equation*}
by using the update rule of momemtum method \eqref{eq:momentum}. Thus
\begin{equation*}
    \|e_{k+1} -e_{k} - \beta(e_{k}-e_{k-1}) + \alpha M_k (e_{k}+\gamma(e_{k}-e_{k-1}))\| \leqslant c_4\alpha^2
\end{equation*}
where $M_k$ is the linear application such that $M_k(a_k-b_k) := \nabla f(a_k) - \nabla f(b_k)$, $M_kx := 0$ for all $x \in \mathrm{span}(a_k-b_k)^\perp$, $a_k := x_{k}+\gamma(x_{k}-x_{k-1})$, $b_k := \bar{x}_{k}+\gamma(\bar{x}_{k}-\bar{x}_{k-1})$ if $a_k \neq b_k$, otherwise $M_k:=0$. Let $v_k = (e_k,e_{k-1})\in\mathbb{R}^{2n}$. We have $\| v_{k+1} - Av_k + \alpha B_kv_k \| \leqslant c_4\alpha^2$ where 
\begin{equation*}
    B_k := 
    \begin{pmatrix}
    1+\gamma & -\gamma \\ 
    0 & \hphantom{-}0
    \end{pmatrix} \otimes M_k.
\end{equation*}
We also have
\begin{equation*}
    \|B_k\| = \left\| \begin{pmatrix}
    1+\gamma & -\gamma \\ 
    0 & \hphantom{-}0
    \end{pmatrix} \right\| \left\| M_k \right\| \leqslant M\sqrt{1+ 2\gamma + 2\gamma^2}
\end{equation*}
since $\bar{x}_{k}+\gamma(\bar{x}_{k}-\bar{x}_{k-1})$ and $x_{k}+\gamma(x_{k}-x_{k-1})$ belong to $S+\gamma(S-S)$. The latter inclusion follows from the induction hypothesis and the fact that $\epsilon \leqslant \delta+1$. Hence $\| v_{k+1} - Av_k + \alpha B_kv_k \|_P \leqslant c_2 c_4\alpha^2$ and thus 
\begin{align*}
    \|v_{k+1}\|_P & \leqslant (\|A\|_P + \alpha \|B_k\|_P)\|v_k\|_P + c_2c_4\alpha^2 \\
    & \leqslant (\|A\|_P + \alpha c_3\|B_k\|)\|v_k\|_P + c_2c_4\alpha^2 \\ 
    & \leqslant (1+c_5 \alpha)\|v_k\|_P + c_2c_4\alpha^2
\end{align*}
where $c_5 = c_3 M\sqrt{1+ 2\gamma + 2\gamma^2}$. By induction, we find that
\begin{align*}
    \|x_{k+1}-\bar{x}_{k+1}\| =& \|e_{k+1}\| \leqslant \|v_{k+1}\| \leqslant c_1^{-1}\|v_{k+1}\|_P \\ 
    \leqslant& c_1^{-1}(1+c_5 \alpha)^k \|v_1\|_P + c_1^{-1}c_2c_4\alpha^2 \sum_{i=0}^{k-1} (1+c_5 \alpha)^i \\
    = & c_1^{-1}c_2(1+c_5 \alpha)^k \|v_1\| + c_1^{-1}c_2c_4\alpha^2 \frac{(1+c_5 \alpha)^k-1}{c_5\alpha} \\
    \leqslant & c_1^{-1}c_2e^{c_5 k\alpha} \|e_1\| + c_1^{-1}c_2c_4c_5^{-1}(e^{c_5 k\alpha}-1)\alpha \\
    \leqslant & c_1^{-1}c_2e^{c_5 T} (\|x_1-x_0\|+\|x_0-\bar{x}_0\| + \|\bar{x}_0-\bar{x}_1\|) \\
    & + c_1^{-1}c_2c_4c_5^{-1}(e^{c_5 T}-1)\alpha \\
    \leqslant & c_1^{-1}c_2e^{c_5 T} (\|\beta(x_0-x_{-1}) - \alpha
    \nabla f(x_0 + \gamma(x_0-x_{-1}))\| \\
    & +L\alpha) + c_1^{-1}c_2c_4c_5^{-1}(e^{c_5 T}-1)\alpha \\
    \leqslant & c_1^{-1}c_2e^{c_5 T} (|\beta| \delta\alpha + L(1-\beta) \alpha
    +L\alpha) \\
     & + c_1^{-1}c_2c_4c_5^{-1}(e^{c_5 T}-1)\alpha \\
    \leqslant & c_1^{-1}c_2[e^{c_5 T} (|\beta|\delta + 2L-L\beta +c_4c_5^{-1}) -c_4c_5^{-1}]\bar{\alpha} \\
    = & \epsilon
\end{align*}
since $\bar{\alpha} = \epsilon c_1c_2^{-1}[e^{c_5 T} (|\beta|\delta + 2L-L\beta +c_4c_5^{-1}) -c_4c_5^{-1})]^{-1}$. 

\section{Proof of Fact \ref{fact:O(alpha)}}
\label{sec:Proof of Fact fact:O(alpha)}

For $k = -1,\ldots,K-1$, we have 
\begin{align*}
x_{k+1} - x_k & = \beta (x_k - x_{k-1}) - \alpha\nabla f (y_k^\gamma) \\
    & = \beta^2 (x_{k-1} - x_{k-2}) - \alpha(\beta\nabla f (y_{k-1}^\gamma)+\nabla f (y_k^\gamma))\\
    &~\vdots\\
    & = \beta^{k+1}(x_0-x_{-1}) - \alpha\sum\limits_{i=0}^{k} \beta^i \nabla f(y_{k-i}^\gamma)
\end{align*}
where $y_k^\gamma := x_k + \gamma(x_k-x_{k-1})$. Let $L$ be a Lipschitz constant of $\bar{\beta}f$ on $S+\gamma(S-S)$ where $S$ is the convex hull of $X$. 
Since $\|x_0-x_{-1}\|\leqslant \delta_0\alpha$, we have
\begin{equation*}
    \|x_{k+1} - x_k\| \leqslant |\beta|^{k+1}\|x_0-x_{-1}\|+\alpha L\bar{\beta}^{-1}\sum_{i=0}^k|\beta|^i \leqslant (\delta_0|\beta|^{k+1}+L)\alpha.
\end{equation*}
Given that $\beta \in (-1,1)$, it suffices to take $\delta := \delta_0+L$. In addition, 
\begin{equation*}
    \|z_k-z_{k-1}\| = (\|x_k-x_{k-1}\|^2+\|x_{k-1}-x_{k-2}\|^2)^{1/2} \leqslant \sqrt{2}\delta\alpha
\end{equation*}
for $k=1,\ldots,K$. 

\section{Proof of Theorem \ref{thm:local_min}}
\label{sec:Proof of Theorem thm:local_min}
It is known that if $F\in C^1(\mathbb{R}^n,\mathbb{R}^n)$, $X$ is an open subset of $\mathbb{R}^n$ such that $\mathrm{rank} (F'(x)) = n$ for all $x\in X$, and $F(X)\subset X$, then for almost every $x \in X$, $F^k(x)$ does not converge as $k\to\infty$ to any fixed point of $F$ in $X$ whose spectral radius is greater than one \cite[Theorem 2]{lee2019first}. The sequence $(F^k)_{k\in \mathbb{N}}$ is defined by $F^{k+1} := F \circ F^k$ for all $k\in \mathbb{N}$ where $F^0:\mathbb{R}^n\rightarrow\mathbb{R}^n$ is the identity. In order to prove Proposition \ref{prop:escapesaddle} below, and ultimately Theorem \ref{thm:local_min}, we relax the assumption that $F(X)\subset X$ and instead only require that $x_k \in X$ for all $k\in \mathbb{N}$. Below, we let $\mu(\cdot)$ and $\rho(\cdot)$ denote the Lebesgue measure and the spectral radius respectively.
 
\begin{lemma}\label{lem:unstable}
If $F\in C^1(\mathbb{R}^n,\mathbb{R}^n)$ and $X$ is an open subset of $\mathbb{R}^n$
such that $\mathrm{rank}(F'(x)) = n$ for all $x\in X$, then
\begin{equation*}
    \mu\left(\left\{x\in \mathbb{R}^n: \forall k \in \mathbb{N},~ F^k(x) \in X,~ \lim_{k\to\infty}F^k(x)\in Y\right\}\right)=0,
\end{equation*}
where $Y := \{ x \in X: F(x) = x, ~ \rho(F'(x)) > 1\}$.
\end{lemma}
\begin{proof}
    Since $F\in C^1(\mathbb{R}^n,\mathbb{R}^n)$ and $\mathrm{rank}(F'(x)) = n$ for all $x\in X$, by the inverse function theorem \cite[Theorem 1 p. 498]{zorich2004mathematical} $F$ is a local diffeomorphism over $X$. By the center-stable manifold theorem \cite[Theorem III.7(2) p. 65]{shub2013global}, for all $y\in Y$, there exists an open neighborhood $B_y$ of $y$ such that its associated local center stable manifold $W^\mathrm{sc}_\mathrm{\rm loc}(y):= \{x\in \mathbb{R}^n: \forall k \in \mathbb{N}, F^k(x) \in B_y\}$ has Lebesgue measure zero. Since $\{B_y:y\in Y\}$ is an open cover of $Y$, by Lindel\"of’s lemma \cite[Theorem 30.3(a)]{munkrestopology} there exists $\{y_i\}_{i\in\mathbb{N}}\subset Y$ such that 
    $Y \subset \cup_{i=0}^\infty B_{y_i}$.
    
    We seek to show that the set
    \begin{equation*}
    W:=\left\{x\in \mathbb{R}^n: \forall k \in \mathbb{N},~ F^k(x) \in X,~ \lim_{k\to\infty}F^k(x)\in Y\right\}    
    \end{equation*}
    has Lebesgue measure zero. In order to do so, we consider the sequence $V_0,V_1,V_2,\hdots : Y \rightrightarrows X$ defined by  $V_0(\cdot):=W^\mathrm{sc}_\mathrm{\rm loc}(\cdot)\cap X$ and $V_{k+1}:=(F_{|X})^{-1}\circ V_k$ for all $k\in\mathbb{N}$ where $F_{|X}$ denotes the restriction of $F$ to $X$. We will show that 
    \begin{equation*}
        W \subset \bigcup_{i=0}^\infty\bigcup_{k=0}^\infty V_k(y_i).
    \end{equation*}
    It is then easy to show by induction that $\mu(V_k(y_i)) = 0$ for all $k,i\in \mathbb{N}$. Indeed, on the one hand $\mu(V_0(y_i))\leqslant \mu(W^\mathrm{sc}_\mathrm{\rm loc}(y_i))=0$. On the other hand, if $\mu(V_k(y_i))=0$, then by \cite[Theorem 1]{ponomarev1987submersions} $\mu(V_{k+1}(y_i))=\mu((F_{|X})^{-1}(V_k(y_i)))=0$ since $\mathrm{rank}(F'(x)) = n$ for all $x\in X$. We conclude that
    \begin{equation*}
        \mu(W)\leqslant \sum_{i=0}^\infty\sum_{k=0}^\infty\mu(V_k(y_i))=0.
    \end{equation*}
    
    Let $x\in W$ and $y:=\lim_{k\to\infty}F^k(x)$. Since $y \in Y \subset \cup_{i=0}^\infty B_{y_i}$, there exists $j\in\mathbb{N}$ such that $y\in B_{y_j}$. Since $B_{y_j}$ is open, there exists $K\in \mathbb{N}$ such that $F^k(x)\in B_{y_j}$ for all $k\geqslant K$, or equivalently, $F^k(F^K(x))\in B_{y_j}$ for all $k\in\mathbb{N}$. Thus $F^K(x)\in W^\mathrm{sc}_\mathrm{\rm loc}(y_j)$ and in fact $F^K(x)\in W^\mathrm{sc}_\mathrm{\rm loc}(y_j)\cap X = V_0(y_j)$. Since $x\in W$ and $F^K(x)\in V_0(y_j)$, we have $F^{K-1}(x)\in F^{-1}(F^K(x)) \cap X= (F_{|X})^{-1}(F^K(x))\subset  (F_{|X})^{-1}(V_0(y_j))=V_1(y_j)$. By induction, it follows that $x\in V_K(y_j)\subset \cup_{i=0}^\infty\cup_{k=0}^\infty V_k(y_i)$. 
\end{proof}

Given an objective function $f\in C^2(\mathbb{R}^n)$ with an $L$-Lipschitz continuous gradient, the momentum method \eqref{eq:momentum} does not converge to any critical point whose Hessian has a negative eigenvalue for almost every initial point if $\alpha\in(0,2(1-\beta)/L)$, $\beta\in(0,1)$ and $\gamma=0$ \cite[Lemma 2]{sun2019heavy}, or $\alpha\in(0,4/L)$, $\beta\in(\max\{0,-1+\alpha L/2\},1)$ and $\gamma=0$ \cite[Theorem 3]{o2019behavior}. In order to prove Theorem \ref{thm:local_min}, we enlarge the set of allowable momentum parameters. Below, we let $\lambda_\mathrm{min}(\cdot)$ denote the minimal real eigenvalue of a matrix.

\begin{proposition}\label{prop:escapesaddle}
    Let $f\in C^{2}(\mathbb{R}^n)$, $X \subset \mathbb{R}^n$ be bounded, $\beta\in(-1,1)\setminus\{0\}$, and $\gamma\in\mathbb{R}$. There exists $\bar{\alpha}>0$ such that for all $\alpha\in(0,\bar{\alpha}]$ and for almost every $(x_{-1},x_0)\in \mathbb{R}^{2n}$, the limit of any convergent sequence $x_{-1},x_0,x_1,\hdots \in X$ generated by the momentum method \eqref{eq:momentum} does not belong to 
    \begin{equation*}
        C^-:= \{ x \in \mathbb{R}^n: \nabla f(x) = 0, \lambda_\mathrm{min}(\nabla^2 f(x)) < 0 \}.
    \end{equation*}
\end{proposition}
\begin{proof}
     Since $X$ is bounded, there exists an open bounded set $\widetilde{X}$ such that $\overline{X}\subset\widetilde{X}$. Let
     $M:=\sup \{\rho(\nabla^2f(x)): x\in \widetilde{X}+\gamma(\widetilde{X}-\widetilde{X})\}<\infty$, $\bar{\alpha}:= |\beta|/(1+|\gamma| M)$, and $\alpha \in (0,\bar{\alpha}]$. Any sequence $x_{-1},x_0,x_1,\hdots \in \mathbb{R}^n$ generated by the momentum method \eqref{eq:momentum} follows the update rule $z_{k+1}=F(z_k)$ for all $k\in\mathbb{N}$ where $z_k:=(x_k,x_{k-1})$ and $F:\mathbb{R}^{2n}\rightarrow\mathbb{R}^{2n}$ is defined by
    \begin{equation*}
      F(x,y) := \begin{pmatrix}
            x +\beta(x-y) - \alpha\nabla f(x+\gamma(x-y)) \\ x
        \end{pmatrix}
    \end{equation*}
    for all $(x,y)\in \mathbb{R}^n \times \mathbb{R}^n$. In order to prove the desired result, we claim that it suffices to check the two following facts:
    \begin{enumerate}
        \item $\mathrm{rank}(F'(z)) = 2n$ for all $z\in \widetilde{X}\times \widetilde{X}$, 
        \item $\{(x,x)\in\mathbb{R}^{2n}:x \in C^-\}\subset Z$,
    \end{enumerate}
    where 
    \begin{equation*}
        Z := \{z\in \mathbb{R}^{2n} : F(z) = z, ~ \rho(F'(z)) > 1\}.
    \end{equation*}
    Indeed, by applying Lemma \ref{lem:unstable} to $F\in C^1(\mathbb{R}^{2n},\mathbb{R}^{2n})$ and the open subset $\widetilde{X}\times \widetilde{X}$ of $\mathbb{R}^{2n}$, it follows that
    for almost every $(x_0,x_{-1}) \in \mathbb{R}^{2n}$, the limit of any convergent sequence $(x_0,x_{-1}), (x_1,x_{0}), \ldots \in X\times X$ such that $F(x_k,x_{k-1})=F(x_{k+1},x_k)$ for all $k \in \mathbb{N}$ does not belong to $Y$ where 
    \begin{align*}
        Y & := Z\cap (\widetilde{X}\times \widetilde{X}) \\
          & \supset \{(x,x)\in\mathbb{R}^{2n}:x \in C^-\} \cap (\widetilde{X}\times \widetilde{X}) \\
          & = \{(x,x)\in\mathbb{R}^{2n}:x \in C^-\cap \widetilde{X}\}.  
    \end{align*}
    In particular, for almost every $(x_0,x_{-1}) \in \mathbb{R}^{2n}$, the limit of any convergent sequence $x_{-1},x_0,x_1, \ldots \in X$ generated by the momentum method \eqref{eq:momentum} does not belong to $C^-\cap \widetilde{X}$. Since $\overline{X}\subset \widetilde{X}$, such a limit must belong to $\widetilde{X}$, and thus does not belong to $C^-$.
    
    We next prove the two facts above. First, for any $(x,y) \in \widetilde{X}\times \widetilde{X}$, we have
    {\small\begin{equation*}
        F'(x,y) = 
        \begingroup 
        \setlength\arraycolsep{5pt}
        \begin{pmatrix}
            (1+\beta)I_n-\alpha(1+\gamma)\nabla^2f(x+\gamma(x-y)) & -\beta I_n +\alpha\gamma\nabla^2f(x+\gamma(x-y)) \\
            I_n & 0
        \end{pmatrix}
        \endgroup
    \end{equation*}}%
where $I_n \in \mathbb{R}^{n\times n}$ is the identity matrix. Since $|\beta| \geqslant \alpha(1+|\gamma|M) > \alpha|\gamma| M$, by \cite[Theorem 3]{silvester2000determinants} for all $(x,y) \in \widetilde{X} \times \widetilde{X}$ we have
    \begin{equation*}
        \mathrm{det}(F'(x,y)) = \mathrm{det}(\beta I_n -\alpha\gamma\nabla^2f(x+\gamma(x-y))) \neq 0
    \end{equation*}
and thus $\mathrm{rank}(F'(x,y)) = 2n$. Second, let $x \in C^-$ and $z := (x,x)$. We seek to show that $z \in Z$. Since $\nabla f(x) = 0$, we have $F(z) = F(x,x) = (x,x) = z$.
    Since $f \in C^2(\mathbb{R}^n)$, $\nabla^2f(x)$ is symmetric and therefore admits an eigendecomposition $\nabla^2f(x)=PDP^{\top}$ where $D=\mathrm{diag}(d_1,\hdots,d_n)$ and $P$ is an orthogonal matrix. Again by \cite[Theorem 3]{silvester2000determinants}, we have  
    \begin{align*}
        \mathrm{det}(\lambda I_{2n}-F'(x,x)) &= \mathrm{det}([\lambda^2-(1+\beta)\lambda+\beta]I_n - [\alpha\gamma-\lambda\alpha(1+\gamma)]\nabla^2f(x)) \\
        &= \mathrm{det}([\lambda^2-(1+\beta)\lambda+\beta]I_n - [\alpha\gamma-\lambda\alpha(1+\gamma)]PDP^{\top}) \\
        &= \mathrm{det}([\lambda^2-(1+\beta)\lambda+\beta]I_n - [\alpha\gamma-\lambda\alpha(1+\gamma)]D)\\
        &= \prod_{i=1}^n ([\lambda^2-(1+\beta)\lambda+\beta] - [\alpha\gamma-\lambda\alpha(1+\gamma)]d_i)\\
        &= \prod_{i=1}^n (\underbrace{\lambda^2+[\alpha(1+\gamma)d_i-(1+\beta)]\lambda+\beta-\alpha\gamma d_i}_{\varphi_i(\lambda):=}).
    \end{align*}
    Since $x\in C^-$, there exists $j\in\{1,\hdots,n\}$ such that $d_j<0$. Since $\varphi_j$ is a quadratic function whose leading coefficient is positive and $\varphi_j(1) = \alpha d_j < 0$, $\varphi_j$ has a root that is greater than 1. Thus $\rho(F'(z))>1$.  
\end{proof}

We are now ready to prove Theorem \ref{thm:local_min}. Let $f\in C^2(\mathbb{R}^n)$ be a semialgebraic function with bounded continuous gradient trajectories. Let $\beta \in (-1,1) \setminus \{0\}$, $\gamma \in \mathbb{R}$, and $\delta \geqslant 0$. Assume that the Hessian of $f$ has a negative eigenvalue at all critical points of $f$ that are not local minimizers. Let $X_0$ be a bounded subset of $\mathbb{R}^n$. By Theorem \ref{thm:convergence}, there exist $\bar{\alpha},c>0$ such that for all $\alpha \in (0,\bar{\alpha}]$, there exists $c_\alpha>0$ such that any sequence $x_{-1},x_0,x_1,\hdots \in \mathbb{R}^n$ generated by the momentum method \eqref{eq:momentum} that satisfies $x_0 \in X_0$ and $\|x_0-x_{-1}\| \leqslant \delta \alpha$ obeys 
\begin{equation*}
        \sum_{i=0}^{\infty} \|x_{i+1}-x_i\| \leqslant c ~~~ \text{and} ~~~ \min_{i=0,\ldots,k}\|\nabla f(x_i)\| \leqslant \frac{c_\alpha}{k+1}, ~ \forall k\in \mathbb{N}.
\end{equation*}
It hence converges to a critical point of $f$ and belongs to the bounded set $B(X_0,c)$. By Proposition \ref{prop:escapesaddle}, after possibly reducing $\bar{\alpha}>0$, for all $\alpha\in(0,\bar{\alpha}]$ and for almost every $(x_{-1},x_0)\in \mathbb{R}^{2n}$, the limit of any convergent sequence $x_{-1},x_0,x_1,\hdots \in B(X_0,c)$ generated by the momentum method \eqref{eq:momentum} is not a critical point of $f$ where the Hessian admits a negative eigenvalue. We conclude for all $\alpha \in (0,\bar{\alpha}]$ and for almost every 
$(x_{-1},x_0) \in \mathbb{R}^n \times X_0$, any sequence $x_{-1},x_0,x_1,\hdots \in \mathbb{R}^n$ generated by the momentum method \eqref{eq:momentum} 
that satisfies $\|x_0-x_{-1}\| \leqslant \delta \alpha$ 
converges to a local minimizer of $f$.

\bibliographystyle{abbrv}    
\bibliography{mybib}

\end{document}